\documentclass[graybox]{svmult}

\usepackage{mathptmx}       
\usepackage{helvet}         
\usepackage{courier}        
\usepackage{type1cm}        
\usepackage{makeidx}         
\usepackage{graphicx}        
\usepackage{multicol}        
\usepackage[bottom]{footmisc}
\smartqed

\makeindex             

\usepackage{amsmath}
\usepackage{amssymb}
\usepackage[english]{babel}
\usepackage[utf8]{inputenc}
\usepackage{setspace}
\allowdisplaybreaks[3]

\newcommand{\AKMSbracket}[1]{\left[#1\right]}
\newcommand{\AKMSpara}[1]{\left(#1\right)}

\newcommand{\K}{\mathbb K}

\newcommand{\A}{\mathcal{A}}
\newcommand{\corps}{\mathbb{K}}
\newcommand{\real}{\mathbb{R}}

\newcommand{\N}{\mathcal{N}}

\newcommand{\dl}{\displaystyle}

\author{ Abdenacer Makhlouf }
\title{On Deformations of  $n$-Lie  algebras}
\institute{ Universit\'e de Haute-Alsace, Laboratoire de math\'ematiques, Informatique et Applications 4 rue des Fr\`eres Lumi\`ere, 68093 Mulhouse, France,\\ \email{abdenacer.makhlouf@uha.fr}
}

\begin{document} 

\maketitle

%
%
\begin{abstract}
The aim of this paper is to review the deformation theory of  $n$-Lie algebras. We summarize the 1-parameter formal deformation theory and provide a generalized approach using any unital commutative associative algebra as a deformation base. Moreover, we discuss  degenerations and quantization of $n$-Lie algebras.

\end{abstract}


\section*{Introduction}
The purpose of this paper  is to provide a survey on deformations of $n$-Lie algebras.
Deformation is one of the oldest techniques used by mathematicians and physicists. The first instances of the so-called deformation theory were given by Kodaira and Spencer for complex structures and by Gerstenhaber for associative algebras. Abstract deformation theory and deformation functors in algebraic geometry were inspired and developed in the works of Andr\'e, Deligne, Goldman, Grothendick, Illusie, Laudal, Lichtenbaum, Milson, Quillen, Schlessinger, and Stasheff. Among concrete deformation theory developed by Gerstenhaber for associative algebras and later with Schack for bialgebras, the Lie algebras case was studied by Nijenhuis and Richardson and then by Fialowski and her collaborators in a more general framework. Deformations of $n$-ary algebras was considered in several papers.
Deformation theory is the study of a family in the neighborhood of a given element. Intuitively, a deformation of a mathematical object is a family of the same kind of objects depending on some parameters. The main and popular tool is the power series ring or more generally any commutative algebras. By standard facts of deformation theory, the infinitesimal deformations of an algebra of a given type are parametrized by a second cohomology of the algebra. More generally, it is stated that deformations are controlled by a suitable cohomology. Deformations help to construct new objects starting from a given object and to infer some of its properties. They can also be useful for classification problems.
A modern approach, essentially due to Quillen, Deligne, Drinfeld, and Kontsevich, is that, in characteristic zero, every deformation problem is controlled by a differential graded Lie algebra, via solutions of Maurer-Cartan equation modulo gauge equivalence.
 Some mathematical formulations of quantization are based on the algebra of observables and consist in replacing the classical algebra of observables (typically complex-valued smooth functions on a Poisson manifold) by a noncommutative one constructed by means of an algebraic formal deformations of the classical algebra. In 1997, Kontsevich solved a longstanding problem in mathematical physics, that is every Poisson manifold admits formal quantization which is canonical up to a certain equivalence.
Deformation theory has been applied as a useful tool in the study of many other mathematical structures in Lie theory, quantum groups, operads, and so on. Even today it plays an important role in many developments of contemporary mathematics, especially in representation theory.

 The $n$-ary algebraic structures, which are natural generalizations of binary operations,  appeared naturally in
various domains of theoretical and mathematical physics. Indeed,
theoretical physics progress of quantum mechanics and the discovery
of the Nambu mechanics  (1973) see \cite{Nambu}, as well as a work
of S. Okubo \cite{Okubo} on Yang-Baxter equation gave  impulse to a
significant development on $n$-ary algebras.  The $n$-ary operations
appeared first through cubic matrices which were introduced in the
nineteenth century by Cayley. The cubic matrices were considered
again and generalized by Kapranov, Gelfand, Zelevinskii in 1994 see
\cite{KapranovGelfandZelinski} and Sokolov in 1972 see
\cite{Sokolov}.
 Another recent motivation to study $n$-ary operation comes for string
theory and M-Branes where  appeared  naturally a so called Bagger-Lambert algebra algebra involving a
ternary operation  \cite{BL2007}.
Hundred of papers are dedicated to Bagger-Lambert algebra.  For
other physical applications see \cite{Kerner, Kerner2, Kerner3,
Kerner4}.

The first conceptual generalization of binary algebras was the
ternary algebras introduced  by Jacobson \cite{Jacobson} in
connection with problems from Jordan theory and quantum mechanics,
he defined the Lie triple systems. A Lie triple system consists of a
space of linear operators on vector space $V$ that is closed under
the ternary bracket $[x,y,z]_T=[[x,y],z]$, where $[x,y]=x y-y x$.
Equivalently, the Lie triple system may be viewed as a subspace of
the Lie algebra closed relative to the ternary product. A Lie triple
system arose also in the study of symmetric spaces \cite{loosSym}.
More generally, we distinguish two kinds of generalization of binary
Lie algebras. Firstly, $n$-ary Lie algebras in which the Jacobi
identity is generalized by considering a cyclic summation over
$\mathcal{S}_{2n-1}$ instead of $\mathcal{S}_3$, see \cite{Hanlon}
\cite{Michor} and secondly $n$-ary Nambu algebras in which the
fundamental identity generalizes the fact that the adjoint maps are
derivations. The corresponding identity is called fundamental identity and it  appeared first in Nambu
mechanics \cite{Nambu}, the abstract definition of $n$-ary Nambu
algebras or $n$-Lie algebras (when the bracket is skew
symmetric) was given by Fillipov in 1985, see \cite{Filippov}  and  \cite{Takhtajan,Takhtajan1} for the algebraic
formulation of the Nambu mechanics. The Leibniz $n$-ary algebras
were introduced and studied in \cite{CassasLodayPirashvili}.

This article is organized as follows. In   the first Section  we summarize the definitions
of $n$-ary  algebras of Lie type and associative type, and provide some classical
examples. Moreover, we discuss the representations of $n$-Lie algebras.  In the second section, we review homological algebra tools and define the cohomology for $n$-Lie algebra that suits with deformation theory. The third section is dedicated to one-parameter formal deformations based on formal power series. We also describe the case where the parameter no longer commutes with the original algebra.  In section 4, we present a more general approach based on any commutative associative algebra, generalizing to $n$-Lie algebras,  the approach developed by Fialowski and her collaborators for Lie algebras. Section 5 deals with algebraic varieties of $n$-Lie algebras and degenerations. In the last section, we discuss $n$-Lie-Poisson algebras and quantization.

\section{Definitions and Examples of $n$-Lie algebras and other types of $n$-ary algebras}
Throughout this paper, $\K$ is a field of characteristic zero and  $\mathcal{N}$
is a  $\K$-vector space.
\subsection{$n$-Lie algebras}
In this section, we provide basics on  $n$-Lie algebras which are also called  Filippov $n$-ary algebras or Nambu-Lie algebras.
\begin{definition}
An $n$-\emph{Lie algebra} is a pair
$(\mathcal{N},[\cdot,\dots,\cdot])$, consisting of a vector space $\mathcal{N}$ and  an
$n$-linear map $[~,\cdots,~]:\mathcal{N}^{\times n}\rightarrow \mathcal{N}$ satisfying
\begin{eqnarray}\label{NambuIdentity}
 [ x_{1},..., x_{ n-1},
[x_{n},...,x_{2n-1}]]=
 \sum_{i=n}^{2n-1}{[x_{n},...,x_{i-1}, [ x_{1},\cdots,x_{n-1},x_{i}],
x_{i+1}, ...,x_{2n-1}],}
\end{eqnarray}%
and \begin{equation}
[x_{\sigma (1)}, \cdots, x_{\sigma (n)}]=Sgn(\sigma )[x_{1}, \cdots,
x_{n}],\quad \forall \sigma \in \mathcal{S}_{n}\text{ and }\forall
x_{1},\cdots,x_{n}\in \mathcal{N}
\end{equation}%
where $\mathcal{S}_{n}$ stands for the permutation group on $n$
elements and $sgn(\sigma)$ denotes the signature of $\sigma$.
\end{definition}
We call  condition (\ref{NambuIdentity})  Nambu
identity, it is also called fundamental identity or Filippov
identity.

\begin{remark}
Let $(\mathcal{N},[\cdot,\dots,\cdot])$ be an $n$-Lie algebra. Let
$x=(x_1,\ldots,x_{n-1})\in \mathcal{N}^{n-1}$ and $y\in \mathcal{N}$. Let $L_x$ be a
linear map on $\mathcal{N}$,
 defined   by  
 \begin{equation}\label{adjointMap}
 L_{x}(y)=[x_{1},\cdots,x_{n-1},y].
 \end{equation}

Then the Nambu identity maybe written
\begin{equation*}
L_{x}( [x_{n},...,x_{2n-1}])= \sum_{i=n}^{2n-1}{[x_{n},...,x_{i-1},
L_{x}(x_{i}), x_{i+1} ...,x_{2n-1}].}
\end{equation*}%
\end{remark}

Let  
$(\mathcal{N},\mu)$ and $(\mathcal{N},\nu)$ be two $n$-ary operations, 
 $\mu,\nu :\mathcal{N}^{ n}\rightarrow \mathcal{N}$. We define a  $(2n-1)$-ary operation $\mu\circ\nu$ by 
\begin{eqnarray}\label{GerstCircle}
 \mu\circ\nu (x_{1},..., x_{ n-1},
x_{n},...,x_{2n-1})=\mu(x_{1},..., x_{ n-1},
\nu(x_{n},...,x_{2n-1}))\\ -
 \sum_{i=n}^{2n-1}{\mu(x_{n},...,x_{i-1}, \nu( x_{1},\cdots,x_{n-1},x_{i}),
x_{i+1}, ...,x_{2n-1}).}
\end{eqnarray}
Then, an $n$-ary operation $\mu$ on a vector space $\mathcal{N}$ satisfies  Nambu identity if and only if $\mu\circ\mu=0$.

Morphisms of $n$-Lie algebras  are defined as
follows.

\begin{definition}
Let $(\mathcal{N},[\cdot,\dots,\cdot])$ and $(\mathcal{N}',[\cdot,\dots,\cdot]')$ be two
$n$-Lie algebras. A linear map $\rho: \mathcal{N}\rightarrow \mathcal{N}'$ is an  $n$-Lie algebras morphism  if it satisfies
\begin{equation*}\rho ([x_{1},\cdots,x_{2n-1}])=
[\rho (x_{1}),\cdots,\rho (x_{2n-1})]' \quad \forall i=1,n-1.
\end{equation*}
\end{definition}

\begin{example}
The polynomial algebra of  $3$ variables $x_{1},x_{2},x_{3},$ with
the bracket defined by the functional jacobian:

\begin{equation}\label{jacobian}
\lbrack f_{1},f_{2},f_{3}]=\left\vert
\begin{array}{ccc}
\frac{\delta f_{1}}{\delta x_{1}} & \frac{\delta f_{1}}{\delta
x_{2}} &
\frac{\delta f_{1}}{\delta x_{3}} \\
\frac{\delta f_{2}}{\delta x_{1}} & \frac{\delta f_{2}}{\delta
x_{2}} &
\frac{\delta f_{2}}{\delta x_{3}} \\
\frac{\delta f_{3}}{\delta x_{1}} & \frac{\delta f_{3}}{\delta
x_{2}} &
\frac{\delta f_{3}}{\delta x_{3}}%
\end{array}%
\right\vert ,
\end{equation}%
is a 3-Lie algebra.
\end{example}

We have also this fundamental example :
\begin{example}

Let $V=\mathbb{R}^4$ be the 4-dimensional oriented euclidian space over $\mathbb{R}$. The bracket of 3 vectors $\overset{\rightarrow }{x_{1}},\overset{\rightarrow }{x_{2}},\overset%
{ \rightarrow }{x_{3}}$   is given by :
\begin{equation*}
\lbrack \overset{\rightarrow }{x_{1}},\overset{\rightarrow }
{x_{2}},\overset{\rightarrow }{x_{3}} ]=\overset{\rightarrow
}{x_{1}}\times \overset{\rightarrow }{x_{2}} \times
\overset{\rightarrow }{x_{3}}=\left\vert
\begin{array}{c}
\begin{array}{cccc}
x_{11} & x_{12} & x_{13} & \overset{\rightarrow }{e_{1}} \\
x_{21} & x_{22} & x_{23} & \overset{\rightarrow }{e_{2}}%
\end{array}
\\
\begin{array}{cccc}
x_{31} & x_{32} & x_{33} & \overset{\rightarrow }{e_{3}}%
\end{array}
\\
\begin{array}{cccc}
x_{41} & x_{42} & x_{43} & \overset{\rightarrow }{e_{4}}%
\end{array}%
\end{array}%
\right\vert
\end{equation*}
where $(x_{1r},...,x_{4r})_{r=1,2,3}$ are the coordinates of  $%
\overset{\rightarrow }{x_{r}}$ with respect to orthonormal basis $\{e_{r}\}$.
Then,  $(V,[.,.,.])$ is a 3-Lie algebra.
\end{example}
\begin{remark}
Every 3-Lie algebra on $\mathbb{R}^4$ could be deduced from the previous example (see \cite{Gautheron1}).
\end{remark}

\subsection{  $n$-ary algebras of associative type}
There are several possible generalizations of binary associative algebras. A typical example is the ternary product of rectangular matrices introduced by Hestenes \cite{Hestenes} defined for  $A,B,C\in \mathcal{M}_{n,m}$ by $AB^*C$ where $B^*$ is the conjugate transpose.

Consider an $n$-ary operation   $ m:\N \otimes \cdots  \otimes \N \rightarrow \N$ or 
 $ m:\N\times \cdots \times \N \longrightarrow \N$. 
  The $n$-ary  operation is said to be
\emph{symmetric} (resp. \emph{skew-symmetric}) if
\begin{equation}
m(x_{\sigma (1)}\otimes \cdots \otimes x_{\sigma (n)})=
m(x_{1}\cdots\otimes x_{n}), \quad\forall \sigma \in \mathcal{S}_{n}\text{ and }%
\forall x_{1},\cdots, x_{n}\in \N,
\end{equation}
resp.
\begin{equation}
m(x_{\sigma (1)}\otimes \cdots\otimes x_{\sigma
(n)})=Sgn(\sigma
)m(x_{1}\otimes \cdots\otimes x_{n}),\quad\forall \sigma \in \mathcal{S}_{n}\text{ and }%
\forall x_{1},\cdots,x_{n}\in \N,
\end{equation}
where $Sgn(\sigma )$ denotes the signature of the permutation
$\sigma\in \mathcal{S}_{n}$.

It  is said to be commutative if
\begin{equation}
\sum_{\sigma \in \mathit{S}_{n}}{\ Sgn(\sigma )
m(x_{\sigma(1)}\otimes \cdots \otimes x_{\sigma (n)}) =0},\
\quad \forall x_{1},\cdots,x_{n}\in \N.
\end{equation}
\begin{remark}
A symmetric ternary operation is commutative.
\end{remark}

We have the following type of "associative" ternary operations.

\begin{definition}
A totally associative $n$-ary algebra is given by a $\mathbb{K}$-vector
space $\N$\ and an $n$-ary operation $m$ satisfying, for all $x_1,\cdots,x_{2n-1}\in \N$,
\begin{equation*}
 m(m(x_1\otimes \cdots \otimes x_n)\otimes \cdots\otimes x_{2n-1}) =m(x_1\otimes\cdots \otimes x_i \otimes m(x_{i+1}\otimes\cdots\otimes x_{i+n})\otimes\cdots  x_{2n-1})\   \forall i.
\end{equation*}
\end{definition}
\begin{example}
Let $\{e_1,e_2\}$ be a basis of a 2-dimensional space  $\N=\mathbb{K
}^2$, the ternary  operation on $\N$ given by
\begin{equation*}
\begin{array}{ll}
\begin{array}{lll}
m(e_{1}\otimes  e_{1}\otimes  e_{1}) & = & e_{1} \\
m(e_{1}\otimes  e_{1}\otimes  e_{2}) & = & e_{2} \\
m(e_{1}\otimes  e_{2}\otimes  e_{2}) & = & e_{1}+e_{2} \\
m(e_{2}\otimes  e_{1}\otimes  e_{1}) & = & e_{2}%
\end{array}
&
\begin{array}{lll}
m(e_{2}\otimes  e_{2}\otimes  e_{1}) & = & e_{1}+e_{2} \\
m(e_{2}\otimes  e_{2}\otimes  e_{2}) & = & e_{1}+2e_{2} \\
m(e_{1}\otimes  e_{2}\otimes  e_{1}) & = & e_{2} \\
m(e_{2}\otimes  e_{1}\otimes  e_{2}) & = & e_{1}+e_{2}%
\end{array}%
\end{array}
\end{equation*}

defines a totally associative ternary algebra.
\end{example}
\begin{definition}
A weak totally associative $n$-ary algebra is given by a
$\mathbb{K}$-vector space $\N$\ and a ternary operation $m$,
satisfying for all 
$x_1,\cdots,x_{2n-1}\in \N$%
,
\begin{eqnarray*}
m(m(x_1\otimes \cdots \otimes x_n)\otimes \cdots\otimes x_{2n-1})  =m(x_1\otimes \cdots \otimes m(x_{n}\otimes \cdots\otimes x_{2n-1})).
\end{eqnarray*}
\end{definition}

Naturally, any totally associative $n$-ary algebra is a weak totally
associative $n$-ary  algebra.

\begin{definition}
A partially associative  $n$-ary algebra is given by a $\mathbb{K}$-vector
space $\N $ and an  $n$-ary operation $m$ satisfying, for all  $x_1,\cdots,x_{2n-1}\in \N$,
\begin{eqnarray*}
\sum_{i=0}^{n-1}m(x_1\otimes\cdots \otimes x_i \otimes m(x_{i+1}\otimes\cdots\otimes x_{i+n})\otimes\cdots  x_{2n-1})=0.
\end{eqnarray*}
\end{definition}

\begin{example}
Let $\{e_1,e_2\}$ be a basis of a 2-dimensional space
  $\N=\mathbb{K }^2$, the ternary  operation on $\N$ given by
$
m(e_{1}\otimes  e_{1}\otimes  e_{1})=e_{2}
$
defines a partially  associative ternary algebra.
\end{example}


\begin{remark}
Let $(\N,\cdot)$ be a bilinear associative algebra. Then, the $n$-ary
operation, defined by
$
m(x_1\otimes \cdots\otimes x_3)=x_1\cdot ... \cdot x_n
$
determines on the vector space $\N$ a structure of totally associative
$n$-ary algebra which is not partially associative.
\end{remark}

The category of totally (resp. partially) $n$-ary algebras is encoded  by non-symmetric operad denoted $tAs^{(n)}$ (resp. $pAs^{(n)}$). The space on $p$-ary non-symmetric operations of  $tAs^{(n)}$ is given by $tAs_{i n-i-1}^{(n)}=\K$, $tAs_{p}^{(n)}=0$ otherwise. If we put the degree $k-2$ on the generating operation of $pAs^{(n)}$, then the non-symmetric  operads $tAs^{(n)}$ and  $pAs^{(n)}$  are Koszul dual to each other. Moreoever, the Koszulity can be proved by the rewriting method \cite{LodayValete-Operad}.\\

There is another generalization of Jacobi condition that leads to another type of $n$-ary Lie algebra.
\begin{definition}
An $n$-ary Lie algebras is a skew-symmetric $n$-ary operation $[~,\cdots,~]$ on a $\mathbb K$-vector space $\N$ satisfying $\forall x_{1},\cdots, x_{2n-1}\in \N$ the
following generalized Jacobi condition
\begin{align*}
\sum_{\sigma \in \mathit{S}_{2n-1}}{\ Sgn(\sigma )[[x_{\sigma
(x_{1})},\cdots,x_{\sigma (x_{n-1})}],
x_{\sigma(x_{n})},\cdots, x_{\sigma (x_{2n-1})}]} =0.
\end{align*}
\end{definition}

As in the binary case, there is a functor which makes correspondence to any
partially associative $n$-ary algebra an $n$-ary Lie algebra (see \cite
{Gnedbaye1,Gnedbaye2}).
\begin{proposition}To any partially associative $n$-ary algebra on a
vector space $\N$ with $n$-ary operation $m$, one associates an
$n$-ary Lie algebra on $\N$ defined $\forall x_{1},\cdots, x_{n}\in \N$ by the bracket
\begin{equation}
[x_{1},\cdots ,x_{n}] =\underset{\sigma \in \mathit{S}_{n}}{\sum Sgn(}\sigma
)m(x_{\sigma (1)}\otimes \cdots \otimes x_{\sigma (n)}).
\end{equation}

\end{proposition}
%
%

\subsection{Representations of $n$-Lie algebras}
In this section we consider  adjoint representations  of $n$-Lie algebras and show that any $n$-Lie algebra can be represented by a Leibniz algebra.
\begin{definition}
A representation of an $n$-Lie algebra $(\mathcal{N},[\cdot ,...,\cdot ])$
on a vector space $\mathcal{N}$ is a skew-symmetric multilinear map $\rho:\mathcal{N}^{ n-1}\longrightarrow End(\mathcal{N})$,
satisfying  for $x,y\in \mathcal{N}^{n-1}$ the identity
\begin{equation}\label{RepIdentity1}
\rho(x)\circ\rho(y)-\rho(y)\circ\rho(x)=\sum_{i=1}^{n-1}\rho(x_1,...,ad_y(x_i),...,x_{n-1}),
\end{equation}

where $ad_y(x_i)=[y_1,\cdots, y_{n-1},x_i]$ is an endomorphism of   $\mathcal{N}$.
\end{definition}
Two representations $\rho$ and $\rho'$ on $\mathcal{N}$ are \emph{equivalent} if there exists $f:\mathcal{N} \rightarrow \mathcal{N} $ an isomorphism of vector space such that $f(x\cdot y)=x\cdot ' f(y)$ where $x\cdot y=\rho(x)(y)$ and $x\cdot' y=\rho'(x)(y)$ for $x\in \mathcal{N}^{n-1}$ and $y\in \mathcal{N}.$
\begin{example}Let $(\mathcal{N}, [\cdot  ,..., \cdot ] )$ be an $n$-Lie algebra.  The map $\rm ad$  defined in \eqref{adjointMap} is a representation. The identity \eqref{RepIdentity1} is equivalent to Nambu identity. It is called adjoint representation.
\end{example}

 Leibniz algebras were introduced by Loday. 
A Leibniz algebra  is a pair $(A, [\cdot, \cdot], \alpha)$
consisting of a vector  space $A$, a bilinear map $[\cdot, \cdot]:
A\times A \rightarrow A$  satisfying, for $x,y,z\in A$, 
\begin{equation} \label{Leibnizalgident}
[x,[y,z]]=[[x,y],z]+[y,[x,z]].
\end{equation}

Let $(\mathcal{N},[\cdot ,...,\cdot ])$ be an $n$-Lie  algebras and  $\wedge^{n-1}\mathcal{N}$ be the set of elements $x_1\wedge...\wedge x_{ n-1}$ that are skew-symmetric in their arguments. We denote by $\mathcal{L}(\mathcal{N})$ the space $\wedge^{n-1}\mathcal{N}$ and  call it  the fundamental set.  Let  $x=x_1\wedge...\wedge x_{ n-1}\in\wedge^{n-1}\mathcal{N},\ y=y_1\wedge...\wedge y_{ n-1}\in\wedge^{n-1}\mathcal{N}, \ z\in \mathcal{N}$.   Let 
   $L:\wedge^{n-1}\mathcal{N}\longrightarrow End(\mathcal{N})$ be a linear map defined  as 
\begin{equation}\label{adj}L(x)\cdot z=[x_1,...,x_{n-1},z],\end{equation}
and extending it linearly to all elements of  $\wedge^{n-1}\mathcal{N}$.  Notice that $L(x)\cdot z=ad_x(z)$. 
We define  a bilinear map $[\ ,\ ]:\wedge^{n-1}\mathcal{N}\times\wedge^{n-1}\mathcal{N}\longrightarrow\wedge^{n-1}\mathcal{N}$  by
\begin{equation}\label{brackLei}[x ,y]=L(x)\bullet y=\sum_{i=0}^{n-1}\big(y_1,...,L(x)\cdot y_i,...,y_{n-1}\big).\end{equation}

\begin{lemma}\label{3.1}
The map $L$ satisfies
\begin{equation} L([x ,y ])\cdot z=L(x)\cdot \big(L(y)\cdot z\big)-L(y)\cdot \big(L(x)\cdot z\big),\end{equation}
for all $x,\ y\in \mathcal{L}(\mathcal{N}),\ z\in \mathcal{N}$.
\end{lemma}
\begin{proposition}\label{HomLeibOfHomNambu}The pair $\big(\mathcal{L}(\mathcal{N}),\ [\ ,\ ]\big)$ is a  Leibniz algebra.
\end{proposition}
\begin{proof}Straightforward verification, see \cite{De Azcarraga3} .
\end{proof}
We obtain a similar result if we consider the space $T\N=\otimes^{n-1} \N$ instead of $\mathcal{L}(\N)$.

\subsection{Central Extensions}

We racall  some basics about extensions of $n$-Lie algebras.
\begin{definition}
Let $A,B,C$ be three $n$-Lie algebras ($n\geq 2$). An extension of $B$ by $A$ is a short sequence: 
\[ A \overset{\lambda}{\to} C \overset{\mu}{\to} B, \]
such that $\lambda$ is an injective homomorphism, $\mu$ is a surjective homomorphism, and\\ $\operatorname{Im} \lambda \subset \ker \mu$. We say also that $C$ is an extension of $B$ by $A$.
\end{definition}
\begin{definition}
Let $A$, $B$ be two $n$-Lie algebras, and  $A \overset{\lambda}{\to} C \overset{\mu}{\to} B$ be an extension of $B$ by $A$.
\begin{itemize}
\item The extension is said to be trivial if there exists an ideal $I$ of $C$ such that\\ $C = \ker \mu \oplus I$.
\item It is said to be central if $\ker \mu \subset Z (C)$.
\end{itemize}
\end{definition}
We may equivalently define central extensions by a $1$-dimensional algebra (we will simply call it central extension) this way:

\begin{definition}
Let $\N$ be an $n$-Lie algebra. We call central extension of $\N$ the space $\bar{\N}=\N\oplus \mathbb{K} c$ equipped with the bracket:
\[ \AKMSbracket{x_1,...,x_n}_c = \AKMSbracket{x_1,...,x_n} + \omega\AKMSpara{x_1,...,x_n} c \text{ and } \AKMSbracket{x_1,...,x_{n-1},c}_c = 0,\forall x_1,...,x_n \in \N, \] 
where $\omega$ is a skew-symmetric $n$-linear form such that $\AKMSbracket{\cdot,...,\cdot}_c$ satisfies the Nambu identity (or Jacobi identity for $n=2$). 
\end{definition}

\begin{proposition}[\cite{De Azcarraga3}]
\begin{enumerate}
\item The bracket of a central extension satisfies the Nambu identity  if and only if $\omega$ is a $2$-cocycle for the scalar cohomology of $n$-Lie algebras.
\item Two central extensions of an $n$-Lie algebra $A$ given by two maps $\omega_1$ and $\omega_2$ are isomorphic if and only if $\omega_2 - \omega_1$ is a $2$-coboundary for the scalar cohomology of $n$-Lie algebras.
\end{enumerate}
\end{proposition}

\section{Deformation cohomology of $n$-Lie algebras}

The basic concepts of homological algebra are those of a complex and
homomorphisms of complexes, defining the category of complexes, see for example \cite{weibel}. A \emph{chain complex} $\mathcal{C} _{.}$ is a sequence $\mathcal{C} =\{\mathcal{C}
_{p}\}_{p}$ of abelian groups or more generally objects of an abelian category and an indexed set $%
\delta=\{\delta_{p}\}_{p}$ of homomorphisms $\delta_{p}:\mathcal{C}
_{p}\rightarrow \mathcal{C} _{p-1}$ such that $\delta_{p-1}\circ \delta_{p}=0$ for all $p$. A chain complex can be considered as a cochain complex by reversing
the enumeration $\mathcal{C} ^{p}=\mathcal{C}
_{-p}$ and $\delta^{p}=\delta_{-p}$. A \emph{cochain complex} $\mathcal{C} $ is a
sequence of abelian groups and homomorphisms
$
\cdots \overset{\delta^{p-1}}{\longrightarrow }\mathcal{C} ^{p}\overset{\delta^{p}}{%
\longrightarrow }\mathcal{C} ^{p+1}\overset{\delta^{p+1}}{\longrightarrow }\cdots
$
with the property $\delta^{p+1}\circ \delta^{p}=0$ for all $p$.
The homomorphisms $\delta^p$ are called \emph{coboundary operators} or \emph{codifferentials}.
 A \emph{cohomology} of a cochain complex $\mathcal{C} $ is given by  the groups $%
H^{p}(\mathcal{C} )=Ker\delta^{p}/Im\delta^{p-1}$.

The elements of $\mathcal{C} ^{p}$ are $p$-cochains, the
elements of $Z^{p}:=Ker \delta^{p}$ are $p$-cocycles, the elements of $%
B^{p}:=Im\delta^{p-1}$ are $p$-coboundaries. Because
$\delta^{p+1}\circ \delta^{p}=0$ for all $p$, we have $0\subseteq
B^{p}\subseteq Z^{p} \subseteq \mathcal{C} ^{p}$ for all $p$. The
$p^{th}$ cohomology group is the quotient $H^p=Z^{p}/B^{p}$.

The cohomology of $n$-Lie algebras is induced by the cohomology of Leibniz algebras.
Let $(\N,[\cdot ,...,\cdot ])$ be an $n$-Lie algebra and the pair $(\mathcal{L}(\N)=\N^{\otimes n-1},[\cdot,\cdot])$  be the Leibniz algebra associated to $\N$ where the bracket is defined in \eqref{brackLei}.
\begin{theorem} Let $(\N,[\cdot ,...,\cdot ])$ be an $n$-Lie algebra. \\ Let $\mathcal{C}^p(\N,\N)=Hom(\otimes^{p}\mathcal{L}(\N)\otimes \N,\N)$ for $p\geq 1$ be the  cochains set and 
$\Delta :\mathcal{C}^p(\N,\N)\rightarrow C^p(\mathcal{L}(\N),\mathcal{L}(\N))$ be the linear map defined for $p=0$ by
\begin{eqnarray*}
\Delta\varphi(x_1\otimes\cdots\otimes x_{n-1})&=& \sum_{i=1}^{n-1}x_1\otimes\cdots\otimes\varphi(x_i)\otimes\cdots\otimes x_{n-1}
\end{eqnarray*} and for $p>0$ by
\begin{align}
&\label{defdelta}(\Delta\varphi )(a_1,\cdots ,a_{p+1})=\sum_{i=1}^{n-1}x_{p+1}^1\otimes\cdots\otimes\varphi(a_1,\cdots ,a_{n-1}\otimes x_{p+1}^i)\otimes\cdots\otimes x_{p+1}^{n-1},\nonumber
\end{align}
where $a_j=x_{j}^1\otimes\cdots\otimes x_j^{n-1}.$
Then there exists a cohomology complex $(\mathcal{C}^\bullet(\N,\N),\delta )$ for $n$-Lie algebras such that $d\circ \Delta =\Delta\circ \delta.$\\

The coboundary map $\delta: \mathcal{C}^p(\N,\N)\rightarrow \mathcal{C}^{p+1}(\N,\N)$ is defined for $\varphi\in \mathcal{C}^p(\N,\N)$ by
\begin{align*}
& \delta^{p+1}\psi(a_1,...,a_p,a_{p+1},z)=
 \sum_{1\leq i< j}^{p+1}(-1)^i\psi\big(a_1,...,\widehat{a_i},...,
a_{j-1},[a_i,a_j],...,a_{p+1},z\big)\nonumber\\
&
+
 \sum_{i=1}^{p+1}(-1)^i\psi\big(a_1,...,\widehat{a_i},...,
a_{p+1},L(a_i)\cdot z\big) + \sum_{i=1}^{p+1}(-1)^{i+1}L(a_i)\cdot \psi\big(a_1,...,\widehat{a_i},...,
a_{p+1},z\big)
\nonumber\\ &
+ (-1)^p\big(\psi(a_1,...,a_p,\ )\cdot a_{p+1}\big)\bullet z,\nonumber
\end{align*}
where
  $$\big(\psi(a_1,...,a_p,\ )\cdot a_{p+1}\big)\bullet z= \sum_{i=1}^{n-1}[x_{p+1}^1,...,\psi(a_1,...,a_p,x_{p+1}^i ),...,x_{p+1}^{n-1},z],$$
  for   element $a_i\in \mathcal{L}(\N)$, $z\in  \N$.
\end{theorem}
In particular, for $p=1$, we get  the set of 2-cocycles 
$$Z^2(\N,\N)=\{ \psi:\N^n\rightarrow   \N \text{ satisfying } \delta^2\psi=0\},$$ where  for $a_1=x_1^1\otimes \cdots \otimes x_1^{n-1}$ and $a_2=x_2^1\otimes \cdots \otimes x_2^{n-1}$
\begin{eqnarray}
&& \delta^2(a_1,a_2,z)=\\
&&-\sum_{i=1}^{n-1}\psi(x_2^1,\cdots ,[x_1^1,\cdots x_1^{n-1},x_2^i],\cdots x_2^{n-1},z)-\psi(x_2^1,\cdots ,x_2^{n-1},[x_1^1,\cdots ,x_1^{n-1},z])\nonumber \\
&& +\psi(x_1^1,\cdots ,x_1^{n-1},[x_2^1,\cdots ,x_2^{n-1},z])+[x_1^1,\cdots ,x_1^{n-1},\psi(x_2^1,\cdots ,x_2^{n-1},z)] \label{cocyle2} \\
&& -[x_2^1,\cdots ,x_2^{n-1},\psi(x_1^1,\cdots ,x_1^{n-1},z)]-\sum_{i=1}^{n-1}[x_2^1,\cdots ,\psi(x_1^1,\cdots x_1^{n-1},x_2^i),\cdots x_2^{n-1},z].\nonumber
\end{eqnarray}
For $p=0$, $\psi:\N\rightarrow \N$ and $(x_1,\cdots,x_n)\in\N^n$, we have
\begin{equation}
\delta^1\psi(x_1,\cdots,x_n)=-\psi([x_1,\cdots,x_n])+\sum_{i=1}^{n}[x_1,\cdots,\psi(x_i)\cdots x_n]
\end{equation}
Notice that a linear map $\psi:\N\rightarrow \N$ such that $\delta^1\psi=0$ is a 1-cocycle and it  corresponds to a derivation of the $n$-Lie algebra. The set of 2-coboundaries is defined as 
$$B^2(\N,\N)=\{\psi :\N^n\rightarrow   \N : \exists \varphi:\N\rightarrow \N \text{ such that } 
\psi=\delta^1 \varphi \}.$$
Hence, the second cohomology group, which plays an important role in deformation theory,  is defined as $$H^2(\N,\N)=Z^2(\N,\N)/B^2(\N,\N).$$

\section{Formal deformation of $n$-Lie algebras }
In this section we study one parameter formal deformations of $n$-Lie algebras. This approach  were introduced by Gerstenhaber for associative \cite{Gerstenhaber64} and by Nijenhuis and Richardson for Lie \cite{NR}. Since then the approach were extended to many other algebraic structures. The main results connect formal deformation to cohomology groups. The noncommutative case was studied by Pincson.

\subsection{One-parameter formal deformation of $n$-Lie algebras }
  Let $\mathbb{K}[[t]]$ be the power series ring in one variable $t$ and coefficients in $\mathbb{K}$ and $\mathcal{N}[[t]]$ be the set of formal series whose coefficients are elements of the vector space $\mathcal{N}$, ($\mathcal{N}[[t]]$ is obtained by extending the coefficients domain of $\mathcal{N}$ from $\mathbb{K}$ to $\mathbb{K}[[t]]$). Given a $\mathbb{K}$-$n$-linear map $\varphi:\mathcal{N}\times...\times \mathcal{N}\rightarrow \mathcal{N}$, it admits naturally an extension to a  $\mathbb{K}[[t]]$-$n$-linear map $\varphi:\mathcal{N}[[t]]\times...\times \mathcal{N}[[t]]\rightarrow \mathcal{N}[[t]]$, that is, if $x_i=\dl\sum_{j\geq0}a_i^jt^j$, $1\leq i\leq n$ then $\varphi(x_1,...,x_n)=\dl\sum_{j_1,...,j_n\geq0}t^{j_1+...+j_n}\varphi(a_1^{j_1},...,a_n^{j_n})$. 

  \begin{definition}Let $(\mathcal{N},[\cdot ,...,\cdot ])$ be an $n$-Lie algebra. A one-parameter formal deformation of the $n$-Lie algebra $\mathcal{N}$ is given by a $\mathbb{K}[[t]]$-$n$-linear map $$[\cdot ,...,\cdot ]_t:\mathcal{N}[[t]]\times...\times \mathcal{N}[[t]]\rightarrow \mathcal{N}[[t]]$$ of the form $[\cdot ,...,\cdot ]_t=\dl\sum_{i\geq0}t^i[\cdot ,...,\cdot ]_i$ where each $[\cdot ,...,\cdot ]_i$ is a skew-symmetric  $\mathbb{K}[[t]]$-$n$-linear map $[\cdot ,...,\cdot ]_i:\mathcal{N}\times...\times \mathcal{N}\rightarrow \mathcal{N}$ (extended  to a  $\mathbb{K}[[t]]$-$n$-linear map), and $[\cdot ,...,\cdot ]_0=[\cdot ,...,\cdot ]$ such that for $(x_i)_{1\leq i\leq 2n-1}$

  \begin{eqnarray}\label{NambuIdentityDefor}
 [ x_{1},..., x_{ n-1},
[x_{n},...,x_{2n-1}]_t]_t=
 \sum_{i=n}^{2n-1}{[x_{n},...,x_{i-1}, [ x_{1},\cdots,x_{n-1},x_{i}]_t,
x_{i+1}, ...,x_{2n-1}]_t}.
\end{eqnarray}%

  The deformation is said to be of order $k$ if $[\cdot ,...,\cdot ]_t=\dl\sum_{i=0}^kt^i[\cdot ,...,\cdot ]_i$ and infinitesimal if $t^2=0$.\\
 The condition \eqref{NambuIdentityDefor}  may be written  for  $x=(x_i)_{1\leq i\leq n-1},\ y=(x_i)_{n\leq i\leq 2n-2}\in \mathcal{L}(\mathcal{N})$ and by setting $z=x_{2n-1}$ 
  \begin{equation}\label{8.1} L_t([x ,y ])\cdot(z)=L_t(x)\cdot \big(L_t(y)\cdot z\big)-L_t(y)\cdot \big(L_t(x)\cdot z\big),\end{equation}
  where $L_t(x)\cdot z=[x_1,...,x_{n-1},z]_t$.\\
 \end{definition}
   Assume that  the deformation is infinitesimal and set $\psi=[\cdot, \cdots,\cdot ]_1$. Then Eq. \eqref{8.1} is equivalent to 
  \begin{eqnarray*}
      & &\big[x_1,....,x_{n-1},\psi(y_1,....,y_n)\big] + \psi\big(x_1,....,x_{n-1},[y_1,....,y_n]\big) 
     \\ &&  =  \sum_{i=1}^n \big[y_1,....,y_{i-1},\psi(x_1,....,x_{n-1},y_i)
  ,y_{i+1},...,y_n\big]  \\
      &&+\sum_{i=1}^n\psi\big(y_1,....,y_{i-1},[x_1,....,x_{n-1},y_i]
  ,y_{i+1},...,y_n\big).
  \end{eqnarray*}
This identity may be viewed as the $2$-cocycle condition $\delta^2[\cdot, ....,\cdot]_1=0$ defined in \eqref{cocyle2}.

More generally, let  
$(\mathcal{N},\mu)$ and $(\mathcal{N},\nu)$ be two $n$-ary operations, 
 $\mu,\nu :\mathcal{N}^{ n}\rightarrow \mathcal{N}$. We define a  $(2n-1)$-ary operation $\mu\circ\nu$ by 
\begin{eqnarray}\label{GerstCircle}
 \mu\circ\nu (x_{1},..., x_{ n-1},
x_{n},...,x_{2n-1})=\mu(x_{1},..., x_{ n-1},
\nu(x_{n},...,x_{2n-1}))\\ -
 \sum_{i=n}^{2n-1}{\mu(x_{n},...,x_{i-1}, \nu( x_{1},\cdots,x_{n-1},x_{i}),
x_{i+1}, ...,x_{2n-1}).}\nonumber
\end{eqnarray}
Then, an $n$-ary operation $\mu$ on a vector space $\mathcal{N}$ satisfies  Nambu identity if and only if $\mu\circ\mu=0$.

Therefore, the Nambu identity \eqref{NambuIdentityDefor} is equivalent to an infinite system, called deformation equation,  
\begin{equation}
\left\{ \sum_{i=0}^{k}[\cdot ,...,\cdot ] _{i}\circ [\cdot ,...,\cdot ]
_{k-i} =0 \qquad k=0,1,2,\cdots \right. 
\end{equation}

For an arbitrary  $k>1,$ the $k^{th}$ equation of the previous system may be written 
\[
\delta^2[\cdot ,...,\cdot ] _{k} =\sum_{i=1}^{k-1}[\cdot ,...,\cdot ] _{i}\circ [\cdot ,...,\cdot ]_{k-i}.
\]
Assume that a deformation of order $m$  satisfies the deformation equation. The
truncated deformation is extended to a deformation of order $m+1$  if 
\[
\delta^2 [\cdot ,...,\cdot ] _{m}=\sum_{i=1}^{m-1}[\cdot ,...,\cdot ] _{i}\circ [\cdot ,...,\cdot ]_{m-i}
.\]
The right-hand side of this equation is called the {\it obstruction } to
find $[\cdot ,...,\cdot ] _{m}$ extending the deformation.

It turns out that the obstruction is a  3-cocycle. Then, if $H^{3}\left(
\N,\N\right) =0$, it follows that all obstructions vanish and every $[\cdot ,...,\cdot ] _{m}\in
Z^{2}\left( \N,\N\right) $ is integrable.\\

In the following, we characterize equivalent and trivial deformations.

\begin{definition} Let $(\mathcal{N},[\cdot ,...,\cdot ])$ be an $n$-Lie algebra. Given two deformations $\mathcal{N}_{t}=(\mathcal{N}[[t]],[\cdot ,...,\cdot ]_{t})$ and $\mathcal{N}^{'}_{t}=(\mathcal{N}[[t]],[\cdot ,...,\cdot ]^{'}_{t})$ of $\mathcal{N}$ where $[\cdot ,...,\cdot ]_{t}=\sum\limits_{i\geq0}^{k}t^{i}[\cdot ,...,\cdot ]_{i}$ and $[\cdot ,...,\cdot ]^{'}_{t}=\sum\limits_{i\geq0}^{k}t^{i}[\cdot ,...,\cdot ]^{'}_{i}$ with $[\cdot ,...,\cdot ]_{0}=[\cdot ,...,\cdot ]^{'}_{0}=[\cdot ,...,\cdot ]$. We say that $\mathcal{N}_{t}$ and $\mathcal{N}^{'}_{t}$ are equivalent if there exists a formal automorphism
$\phi_{t}:\mathcal{N}[[t]] \longrightarrow  \mathcal{N}[[t]]$ that may be written in the form $\phi_{t}=\sum\limits_{i\geq0}\phi_{i}t^{i}$, where $\phi_{i}\in End(\mathcal{N})$ and $\phi_{0}=Id$ such that
\begin{eqnarray}\label{dAN}
    \phi_{t}([x_1,\cdots ,x_n]_{t})&=&[\phi_{t}(x_1),\cdots , \phi_{t}(x_n)]^{'}_{t}.
\end{eqnarray}
\end{definition}
A deformation $\mathcal{N}_{t}$ of $\mathcal{N}$ is said to be trivial if  $\mathcal{N}_{t}$ is  equivalent to $\mathcal{N}$, viewed as an $n$-ary  algebra on  $\mathcal{N}[[t]]$.

Let $(\mathcal{N},[\cdot ,...,\cdot ])$ be an $n$-Lie  algebra and $[\cdot ,...,\cdot ]_{1}\in Z^{2}(\mathcal{N},\mathcal{N})$.\\
The $2$-cocycle $[\cdot ,...,\cdot ]_{1}$ is said to be \emph{integrable} if there exists a family $([\cdot ,...,\cdot ]_{i})_{i \geq 0}$ such that $[\cdot ,...,\cdot ]_{t}=\sum\limits_{i\geq0}t^{i}[.,.]_{i}$
defines a formal deformation $\mathcal{N}_{t}=(\mathcal{N},[\cdot ,...,\cdot ]_{t})$ of $\mathcal{N}$.

\begin{theorem}
Let $(\mathcal{N},[\cdot ,...,\cdot ])$ be an $n$-Lie algebra and  $(\mathcal{N}[[t]],[\cdot ,...,\cdot ]_t)$, where $[\cdot ,...,\cdot ]_t=\dl\sum_{i\geq0}t^i[\cdot ,...,\cdot ]_i$, be a one-parameter formal deformation. 
\begin{enumerate}
\item The first term  $[\cdot, ....,\cdot]_1$ is a 2-cocycle, that is $[\cdot, ....,\cdot]_1\in Z^2(\N,\N)$.
\item There exists an equivalent deformation $\mathcal{N}^{'}_{t}=(\mathcal{N},[\cdot ,...,\cdot ]^{^{'}}_{t})$, where $[\cdot ,...,\cdot ]^{'}_{t}=\sum\limits_{i\geq 0}t^{i}[\cdot ,...,\cdot ]^{'}_{i}$ such that $[\cdot ,...,\cdot ]^{'}_{1} \in Z^{2}(\mathcal{N},\mathcal{N})$ and
    $[\cdot ,...,\cdot ]^{'}_{1}\not \in B^{2}(\mathcal{N},\mathcal{N})$.\\
    Moreover, if $H^{2}(\mathcal{N},\mathcal{N})=0$, then every one-parameter formal deformation is trivial.
\end{enumerate}
\end{theorem}
The proof is similar to the case $n=2$.

\subsection{Noncommutative one-parameter formal deformations}
In previous formal deformation theory, the  parameter commutes with the
original algebra. Motivated by some nonclassical deformation appearing in
quantization of Nambu mechanics, Pinczon introduced a deformation called
noncommutative deformation where the parameter no longer commutes with the
original algebra. He developed also the associated cohomology \cite{pincson}.

Let $\N$ be a $\corps$-vector space and $\sigma $ be an endomorphism of $\N$. We
give $\N[[t]]$ a $\corps [[t]]$-bimodule structure defined for every $a_{p}\in
\A,\lambda _{q}\in \corps$ by : 
\begin{eqnarray*}
\sum_{p\geq 0}a_{p}t^{p}\cdot \sum_{q\geq 0}\lambda _{q}t^{q}
&=&\sum_{p,q\geq 0}\lambda _{q}a_{p}t^{p+q}, \\
\sum_{q\geq 0}\lambda _{q}t^{q}\cdot \sum_{p\geq 0}a_{p}t^{p}
&=&\sum_{p,q\geq 0}\lambda _{q}\sigma ^{q}(a_{p})t^{p+q}.
\end{eqnarray*}

\begin{definition}
A $\sigma $-deformation of an $n$-ary algebra $\N$ is a $\corps$-algebra structure on $%
\N[[t]]$ which is compatible with the previous $\corps[[t]]$-bimodule structure
and such that $$\N[[t]]/(\N[[t]]t)\cong \A.$$
\end{definition}
A generalization of these deformations was proposed  by F. Nadaud \cite{Nadaud1} where
he considered deformations based on two commuting endomorphisms $\sigma $ and 
$\tau $. The $\corps[[t]]$-bimodule structure on $\N[[t]]$ is defined for $a\in \N$
by the formulas $t\cdot a=\sigma (a)t$ and $a\cdot t=\tau (a)t$, ($a\cdot t$
being the right action of $t$ on $a$).

The remarquable difference with commutative deformation is that the Weyl
algebra of differential operators with polynomial coefficients over $\real$ is
rigid for commutative deformations but has a nontrivial noncommutative
deformation; it is given by the enveloping algebra of the Lie algebra
$osp(1,2)$.

\section{Global deformations}
This approach follows from a general fact in Schlessinger's works \cite{schliss1} 
and was developed by A. Fialowski and her collaborators for different kind of algebras (Lie algebra, Leibniz algebras ....\cite{Fialowski88,Fialowski90,Fialowski99,Fialowski02,Fialowski-Operad}).  In the sequel we extend this approach to $n$-Lie algebras. Let $B$ be a commutative 
algebra over a field $\corps$ of characteristic 0 and an
augmentation morphism $\varepsilon :\mathcal{A}\rightarrow \corps$ (a $\corps$-algebra homomorphism, $%
\varepsilon (1_{B})=1$). We set $m_{\varepsilon }=Ker(\varepsilon )$; $%
m_{\varepsilon }$ is a maximal ideal of $B$. A maximal ideal $m$ of $%
B $ such that $\A/m\cong \corps$, defines naturally an augmentation. We call $(B,m)$
base of deformation.

\begin{definition}
A global deformation of base $(B,m)$ of an $n$-Lie algebra $(\mathcal{N},[\cdot ,...,\cdot ]) $  is a structure of $B$-algebra on the tensor product $%
B\otimes _{\corps}\mathcal{N}^n $ with a bracket $[\cdot ,...,\cdot ]_{B }$ such that $%
\varepsilon \otimes id:B\otimes \mathcal{N} \rightarrow \corps\otimes \mathcal{N} =%
\mathcal{N} $ is an $n$-ary algebra homomorphism. i.e. $\forall a,b\in B$ and $\forall
x,y\in \mathcal{A} $ :

\begin{enumerate}
\item  $[a_1\otimes x_1,\cdots, a_n\otimes x_n]_B=(a_1...a_n\otimes id)[1\otimes x_1,\cdots, 1\otimes x_n ] _{B}\quad (B-$linearity$)$
\item  The bracket $ [\cdot ,...,\cdot ]_{B }$ satisfies Nambu identity.
\item  $\varepsilon \otimes id\left([1\otimes x_1,\cdots, 1\otimes
x_n]_B\right) =1\otimes [x_1 ,...,x_n ] $
\end{enumerate}
\end{definition}

Every formal deformation of an $n$-Lie algebra $\mathcal{N}$, in Gerstenhaber sense, is a global deformation with
a basis $\left( B,m\right) $ where $B=\corps[[t]]$ and $m=t \corps[[t]]$.

\begin{remark}
Condition 1 shows that to describe a global deformation it is enough to know
the brackets $[1\otimes x_1 ,...,1\otimes x_n ] _{B },$ where $x_1,\cdots, x_n\in \mathcal{N} .
$ The conditions 1 and 2 show that we have an $n$-Lie algebra and the last
condition insures the compatibility with the augmentation. We deduce 
\[
[1\otimes x_1,\cdots, 1\otimes x_n]_B=1\otimes [x_1 ,...,x_n ]  +\sum_{i}\alpha
_{i}\otimes z_{i}\quad \text{with }\alpha _{i}\in m,\ z_{i}\in \mathcal{N}.
\]
\end{remark}

\begin{itemize}
\item A global deformation is called {\em trivial\em} if the structure of $n$-ary $B$-algebra on $B\otimes _{\corps}\mathcal{N} $ satisfies $[1\otimes x_1 ,...,1\otimes x_n ]  _{B }=1\otimes [x_1 ,...,x_n ].$
\item Two deformations of an $n$-Lie algebra with the same base are called {\em equivalent \em} (or
isomorphic) if there exists an algebra isomorphism between the two copies of 
$B\otimes _{\corps}\mathcal{N^n}$, compatible with $\varepsilon \otimes id.$
\item A global deformation with base $(B,m)$ is called {\em local \em} if $B$ is a local $\corps$
-algebra with a unique maximal ideal $m_{B}$. If, in addition $m^{2}=0$, 
the deformation is called {\em infinitesimal \em}.
\item Let $B^{\prime }$ be another commutative algebra over $\corps$ with augmentation $%
\varepsilon ^{\prime }:B^{\prime }\rightarrow \corps$ and $\Phi :B\rightarrow
B^{\prime }$ an algebra homomorphism such that $\Phi (1_{B})=1_{B^{\prime }}$
and $\varepsilon ^{\prime }\circ \Phi =\varepsilon $. If a deformation $%
\mu_{B} $ with a base $(B,Ker(\varepsilon ))$ of $\mathcal{A} $ is given we call
push-out $[\cdot ,...,\cdot ]_{B ^{\prime }}=\Phi _{*}[\cdot ,...,\cdot ]_{B} $ a deformation of $%
\mathcal{A} $ with a base $(B^{\prime },Ker(\varepsilon ^{\prime }))$ with the
following algebra structure on $B^{\prime }\otimes \mathcal{A} %
=\left( B^{\prime }\otimes _{B}B\right) \otimes \mathcal{A} =B^{\prime }\otimes
_{B}\left( B\otimes \mathcal{A} \right) $ 
$$[a_{1}^{\prime}\otimes _{B}\left( a_{1}\otimes x_{1}\right) ,\cdots , a_{n}^{\prime }\otimes
_{B}\left( a_{n}\otimes x_{n}\right) ]_{B ^{\prime }} :=a_{1}^{\prime }....a_{n}^{\prime
}\otimes _{B}[ a_{1}\otimes x_{1},\cdots , a_{n}\otimes x_{n}]_{B }, $$
 with $a_{1}^{\prime },a_{2}^{\prime }\in B^{\prime
},a_{1},a_{2}\in B,x_{1},x_{2}\in \mathcal{A} $. The algebra $B^{\prime }$ is viewed
as a $B$-module with the structure $aa^{\prime }=a^{\prime }\Phi \left(
a\right)$. Suppose that  $$[1\otimes
x_1,\cdots, 1\otimes x_n]_B=1\otimes [x_1,\cdots,x_n]+\sum_{i}\alpha _{i}\otimes z_{i}\quad $$ with 
$\alpha _{i}\in m,\ z_{i}\in \mathcal{N} $. Then $$[1\otimes
x_1,\cdots, 1\otimes x_n]_{B^{\prime
}}=1\otimes \mu [
x_1,\cdots, x_n]+\sum_{i}\Phi (\alpha
_{i})\otimes z_{i}\quad $$ with $\alpha _{i}\in m,\ z_{i}\in \mathcal{N} $.
\end{itemize}

One may address the problem of finding, for a fixed algebra, particular deformations which induces all the others in 
the space of all deformations (moduli space) or in a fixed
category of deformations.    The problem of constructing universal or versal deformations of Lie algebras was considered 
for the categories of deformations over infinitesimal local algebras and complete 
local algebras 
(see
 \cite{Fialowski88},\cite{Fialowski99},
 \cite{Fialowski02}). They show that if we consider the
infinitesimal deformations, i.e. the deformations over local algebras $B$ such that
$m_{B}^{2}=0$ where  $m_{B}$ is the maximal ideal, then there exists a universal
deformation (the morphism between base algebras is unique). If we consider the category of complete local rings, then there does
not exist a universal deformation but only versal deformation (there is no unicity for the
morphism).

Let $B$ be a complete local algebra over $\corps$, so $B=$ 
$\overleftarrow{\lim }_{n\rightarrow \infty }(B/m^{n})$  (inductive limit), where $m$ is the
maximal ideal of $B$ and we assume that $B/m\cong \corps$.

A formal global deformation of $\mathcal{N}$ with base $\left( B,m\right) $ is an
algebra structure on the completed tensor product $B\stackrel{\wedge }{%
\otimes }\mathcal{N} =\overleftarrow{\lim }_{n\rightarrow \infty }
((B/m^{n})\otimes \mathcal{N} )$ such that $\varepsilon \stackrel{\wedge }{%
\otimes }id:B\stackrel{\wedge }{\otimes }\mathcal{N} \rightarrow K\otimes 
\mathcal{N} =\mathcal{N} $ is an algebra homomorphism.

The formal global deformation of $\mathcal{N} $ with base $\left(
\corps [[t]],t\corps [[t]]\right) $ are the same as formal one parameter deformation
of Gerstenhaber.

%



\section{The algebraic varieties $\mathcal{L}ie^n_{m}$  and Degenerations }

Let $\N $ be an
$m$-dimensional vector space over  $\K$ and $\{e_{1},\cdots ,e_{m}\}$ be a basis
of $\N$. An $n$-linear bracket  $[\cdot ,...,\cdot ] $  can be defined  by specifying the $m^{n+1}$
structure constants $C_{i_1,\cdots, i_n}^{k}\in \corps$ where 
$$[e_{i_1} ,...,e_{i_m} ]=\sum_{k=1}^{m}C_{i_1,\cdots, i_n}^{k}e_{k}.$$ The Nambu identity and skew-symmetry 
 limits the sets of structure constants 
 $C_{i_1,\cdots, i_n}^{k}$ to a subvariety of $\K ^{m^2 (m-1)\cdots (m-n+1)}$ which we denote by $\mathcal{L}ie^n_{m}$. 
 It is generated by the polynomial relations
\begin{eqnarray}\label{EquationVarNambu}
\sum_{k=1}^{m}{C_{j_1,\cdots , j_n}^{k} C_{i_1,\cdots, i_{n-1},k}^{s}}-
\sum_{r=1}^{n}\sum_{k=1}^{m}{C_{i_1,\cdots , i_{n-1},j_r}^{k}C_{j_1,\cdots , j_{r-1},k, j_{r+1},\cdots  j_{n}}^{s}}=0, \\ \quad 1\leq i_1\cdots , i_{n-1},j_1,\cdots,j_n,s\leq m. \nonumber
\end{eqnarray}
Therefore, $\mathcal{L}ie^n_{m}$ carries a structure of algebraic variety which  is quadratic, non regular and in general non-reduced. The
natural action of the group $GL_m(\corps)$ corresponds to the change of basis :
two $n$-Lie algebras  $(\N,[\cdot ,...,\cdot ]  _{1})$ and $(\N,[\cdot ,...,\cdot ]  _{2})$ are isomorphic if there exists 
$f$ in $GL_m(\corps)$ such that $\N_2=f\cdot \N_1$, that is : 
\[
\forall x_1,\cdots, x_n\in \N\quad [x_1 ,...,x_n ]_2 =f^{-1}([f(x_1) ,...,f(x_n) ]_1.
\]

The orbit of an $n$-Lie algebra $\N_0=(\N,[\cdot ,...,\cdot ]_{0})$, denoted by $\vartheta \left( \N
_{0}\right) $, is the set of all its isomorphic $n$-Lie algebras.


A point in $\mathcal{L}ie^n_{m}$ is defined by $m^2 (m-1)\cdots (m-n+1)$ parameters, which are the structure constants
$C_{i_1,\cdots, i_n}^{k}$ satisfying \eqref{EquationVarNambu}.
The orbits  are in 1-1-correspondence with the isomorphism
classes of $m$-dimensional $n$-Lie algebras. 
The stabilizer subgroup of $\N _0$ 
$$ stab\left( \N_0\right) =\left\{ f\in
GL_m\left( \corps\right) :\N_0=f\cdot \N_0\right\} $$ is  $Aut\left(
\N_0\right) $, the automorphism group of $\N_0$. 
The orbit $\vartheta \left( \N_0\right) $ is identified with the homogeneous space $GL_m\left(
\corps\right) /Aut\left( \N_0\right) $. Then 
$$
\dim \vartheta \left( \N_0\right) =m^2-\dim Aut\left( \N_0\right).
$$
The orbit $\vartheta \left( \N_0\right) $ is provided, when $\corps=\mathbb{C}$ (a complex
field), with the structure of a differentiable manifold. In fact, $\vartheta
\left( \N_0\right) $ is the image through the action of the Lie group $%
GL_{m}\left( \corps\right) $ of the point $\N_0$, considered as a point of $%
Hom\left( \N^{\otimes n} ,\N \right)$. The Zariski tangent space to $\mathcal{L}ie^n_{m}$ at the point $\N_0$
corresponds to $Z^{2}(\N,\N)$ and the tangent space to the orbit corresponds 
to $B^{2}(\N ,\N)$.

The first approach to  study  varieties $\mathcal{L}ie^n_{m}$ is to
establish classifications of $n$-Lie  algebras up to isomorphisms for a fixed
dimension. Classification of $n$-Lie algebras of dimension less than or equal to $n + 2$ is known, see \cite{Filippov, BSZ}. We have the following results.

\begin{theorem}[\cite{Filippov}] \label{AKMSd4}
Any $n$-Lie algebra $\N$ of dimension less than or equal to $n+1$ is isomorphic to one of the following $n$-ary algebras: (omitted brackets are either obtained by skew-symmetry or  $0$)
\begin{enumerate}
\item If $dim \N < n$ then $A$ is abelian.
\item If $dim \N = n$, then we have 2 cases:
\begin{enumerate}
\item $A$ is abelian.
\item $\AKMSbracket{e_1,...,e_n} = e_1.$
\end{enumerate}
\item if $dim \N = n+1$ then we have the following cases:
\begin{enumerate}
\item $A$ is abelian.
\item $\AKMSbracket{e_2,...,e_{n+1}} = e_1$.
\item $\AKMSbracket{e_1,...,e_n}=e_1$.
\item $\AKMSbracket{e_1,...,e_{n-1},e_{n+1}} = a e_n + b e_{n+1} ; \AKMSbracket{e_1,...,e_n} = c e_n + d e_{n+1}$, with $C=
\begin{pmatrix}
a & b \\
c & d
\end{pmatrix}$ an invertible matrix. Two such algebras, defined by matrices $C_1$ and $C_2$, are isomorphic if and only if there exists a scalar $\alpha$ and an invertible matrix $B$ such that $C_2 = \alpha B C_1 B^{-1}$.
\item $\AKMSbracket{e_1,...,\widehat{e_i},...,e_n}= a_i e_i$ for $ 1 \leq i \leq r$, $2<r= \operatorname{dim} D^1(A)\leq n$, $a_i \neq 0$
\item  $\AKMSbracket{e_1,...,\widehat{e_i},...,e_n}= a_i e_i$ for $ 1 \leq i \leq n+1$ which is simple.
\end{enumerate}

\end{enumerate}
\end{theorem}

\begin{theorem}[\cite{BSZ} ] \label{AKMSd5}
Let $\mathbb{K}$ be an algebraically closed field. Any $(n+2)$-dimensional $n$-Lie algebra $\N$  is isomorphic to one of the $n$-ary algebras listed below, where $\N^1$ denotes $\AKMSbracket{\N,...,\N}$:

\begin{enumerate}
\item If $dim \N^1 = 0$ then $\N$ is abelian.
\item If $dim \N^1 = 1$, let $\N^1=\langle e_1 \rangle$, then we have 
\begin{enumerate}
\item $\N^1 \subseteq Z(\N)$ : $\AKMSbracket{e_2,...,e_{n+1}}=e_1$.
\item $\N^1 \nsubseteq Z(\N)$ : $\AKMSbracket{e_1,...,e_n}=e_1$.
\end{enumerate}
\item If $dim \N^1 = 2$, let $\N^1 = \langle e_1,e_2 \rangle$, then we have 
\begin{enumerate}
\item $\AKMSbracket{e_2,...,e_{n+1}}=e_1 ; \AKMSbracket{e_3,...,e_{n+2}}=e_2$.
\item $\AKMSbracket{e_2,...,e_{n+1}}=e_1 ; \AKMSbracket{e_2,e_4,...,e_{n+2}}=e_2 ; \AKMSbracket{e_1,e_4,...,e_{n+2}}=e_1$.
\item $\AKMSbracket{e_2,...,e_{n+1}}=e_1 ; \AKMSbracket{e_1,e_3,...,e_{n+1}}=e_2$.
\item $\AKMSbracket{e_2,...,e_{n+1}}=e_1 ; \AKMSbracket{e_1,e_3,...,e_{n+1}}=e_2 ; \AKMSbracket{e_2,e_4,...,e_{n+2}}=e_2 ;\\ \AKMSbracket{e_1,e_4,...,e_{n+2}}=e_1$.
\item $\AKMSbracket{e_2,...,e_{n+1}}= \alpha e_1 + e_2 ; \AKMSbracket{e_1,e_3,...,e_{n+1}}=e_2$.
\item $\AKMSbracket{e_2,...,e_{n+1}}=\alpha e_1+e_2 ; \AKMSbracket{e_1,e_3,...,e_{n+1}}=e_2 ; \AKMSbracket{e_2,e_4,...,e_{n+2}}=e_2 ;\\ \AKMSbracket{e_1,e_4,...,e_{n+2}}=e_1$.
\item $\AKMSbracket{e_1,e_3,...,e_{n+1}}=e_1 ; \AKMSbracket{e_2,e_3,...,e_{n+1}}=e_2$.
\end{enumerate}
where $\alpha \in \mathbb{K} \setminus \left\{0\right\}$
\item If $dim \N^1 = 3$, let $\N^1=\langle e_1,e_2,e_3 \rangle$, then we have 
\begin{enumerate}
\item $\AKMSbracket{e_2,...,e_{n+1}}=e_1 ; \AKMSbracket{e_2,e_4,...,e_{n+2}}=-e_2 ; \AKMSbracket{e_3,...,e_{n+2}}=e_3$.
\item $\AKMSbracket{e_2,...,e_{n+1}}=e_1 ; \AKMSbracket{e_3,...,e_{n+2}}=e_3 + \alpha e_2 ; \AKMSbracket{e_2,e_4,...,e_{n+2}}=e_3 ; \AKMSbracket{e_1,e_4,...,e_{n+2}}=e_1$.
\item $\AKMSbracket{e_2,...,e_{n+1}}=e_1 ; \AKMSbracket{e_3,...,e_{n+2}}=e_3 ; \AKMSbracket{e_2,e_4,...,e_{n+2}}=e_2 ; \AKMSbracket{e_1,e_4,...,e_{n+2}}=2e_1$.
\item $\AKMSbracket{e_2,...,e_{n+1}}=e_1 ; \AKMSbracket{e_1,e_3,...,e_{n+1}}=e_2 ; \AKMSbracket{e_1,e_2,e_4,...,e_{n+1}}=e_3$.
\item $\AKMSbracket{e_1,e_4,...,e_{n+2}}=e_1 ; \AKMSbracket{e_2,e_4,...,e_{n+2}}=e_3 ; \AKMSbracket{e_3,...,e_{n+2}}=\beta e_2+(1+\beta)e_3$,\\ $\beta \in \mathbb{K} \setminus \left\{0,1\right\}$.
\item $\AKMSbracket{e_1,e_4,...,e_{n+2}}=e_1 ; \AKMSbracket{e_2,e_4,...,e_{n+2}}=e_2 ; \AKMSbracket{e_3,...,e_{n+2}}=e_3$.
\item $\AKMSbracket{e_1,e_4,...,e_{n+2}}=e_2 ; \AKMSbracket{e_2,e_4,...,e_{n+2}}=e_3 ; \AKMSbracket{e_3,...,e_{n+2}}=s e_1 + t e_2 + u e_3$. And $n$-Lie algebras corresponding to this case with coefficients $s,t,u$ and $s',t',u'$ are isomorphic if and only if there exists a non-zero element $r \in \K$ such that 
\[ s=r^3 s' ; t=r^2 t' ; u=r u' .\]
\end{enumerate}
\item If $dim \N^1 = r$ with $4 \leq r \leq n+1$, let $A^1 = \langle e_1,e_2,...,e_r \rangle$, then we have 
\begin{enumerate}
\item $\AKMSbracket{e_2,...,e_{n+1}}=e_1 ; \AKMSbracket{e_3,...,e_{n+2}}=e_2 ; ... ;  \AKMSbracket{e_2,...,e_{i-1},e_{i+1},...,e_{n+2}}=e_i ; \\ \AKMSbracket{e_2,...,e_{r-1},e_{r+1},...,e_{n+2}}=e_r$.
\item $\AKMSbracket{e_2,...,e_{n+1}}=e_1 ; ... ; \AKMSbracket{e_1,...,e_{i-1},e_{i+1},e_{n+1}}=e_i ; ... ; \\  \AKMSbracket{e_1,...,e_{r-1},e_{r+1},e_{n+1}}=e_r$.
\end{enumerate}
\end{enumerate}
\end{theorem}

\

The second approach to study the algebraic variety $\mathcal{L}ie^n_{m}$ is to describe its irreducible components. This problem was considered for binary Lie algebras of small dimensions but it is still open for $n$-Lie algebras. The main approach
uses formal deformations and degenerations.  A degeneration notion is a sort of dual notion of a deformation. It appeared  first in physics literature (Inonu and Wigner
1953 \cite{I_W}). Degeneration is also called specialisation or contraction.
We provide first the geometric definition of a degeneration, using Zariski topology.
\begin{definition}
Let $\N _{0}=(\N,[\cdot ,...,\cdot ]  _{0})$ and $\N _{1}=(\N,[\cdot ,...,\cdot ]  _{1})$ be two $m$-dimensional $n$-Lie algebras.
We said that $\N _{0}$ is a degeneration of $\N_{1}$ if  $\N _{0}$  belongs to the closure of the 
orbit of $\N_{1}$ in $\mathcal{L}ie^n_{m}$ ($\N _{0}\in\overline{\vartheta \left( \N_1\right)}$).

Therefore, $\N_{0}$ and  $\N _{1}$ are in the same irreducible component. 
\end{definition}

A characterization of  degeneration for Lie algebras, in the global deformations viewpoint,  was given by Grunewald and O'Halloran in
 \cite{Grun_OHallo}.  It generalizes naturally to $n$-Lie as follows.
 \begin{theorem}
 Let $\mathcal{N}_0$ and $\mathcal{N}_1$ be two $m$-dimensional $n$-Lie
algebras over $\corps$ with brackets $[\cdot ,...,\cdot ]_0$ and $[\cdot ,...,\cdot ]_1$. The
$n$-Lie algebra $\mathcal{N}_0$ is a degeneration of $\mathcal{N}_1$ if and only if
there is a discrete valuation $\mathbb{K}$-algebra $B$ with residue field $\corps$ whose quotient field
$\mathcal{K}$ is
finitely generated over $\corps$ of transcendence degree one (one parameter), and there is an $m$-dimensional $n$-Lie algebra $[\cdot ,...,\cdot ]
_{B}$ over $B$ such that
 $[\cdot ,...,\cdot ]_B \otimes \mathcal{K} \cong [\cdot ,...,\cdot ]_1 \otimes \mathcal{K}$ and $[\cdot ,...,\cdot ]_B \otimes \corps \cong [\cdot ,...,\cdot ]_0$.
 \end{theorem}
 We call such a degeneration, a global degeneration. A formal degeneration is defined as follows.
 
\begin{definition}
Let  $\N _{1}=(\N,[\cdot ,...,\cdot ]  _{1})$ be an $m$-dimensional  $n$-Lie algebra.
Let $t$ be a parameter in $\corps$ and $\{f_{t}\}_{t\neq 0}$\ be a continuous 
family of  invertible linear maps on $\N$\ over $\corps$. 

The limit (when it exists) of a sequence $f_{t}\cdot \N_{1}$\ , $%
\N_{0}=\lim_{t\rightarrow 0}f_{t}\cdot \N_{1}$, is a {\em formal degeneration \em} of $\N_{1}$ in the
sense that $\N _{0}$\ is in the Zariski closure of the set $\left\{ f_{t}\cdot
\N_1\right\} _{t\neq 0}.$
\\  The bracket  $[\cdot ,...,\cdot ]  _{0}$ is given by 
$$
\forall x_1,\cdots, x_n\in \N\quad [x_1 ,...,x_n ]_0 =\lim_{t\rightarrow 0}f_t^{-1}([f_t(x_1) ,...,f_t(x_n) ]_1.
$$
\end{definition}
We have the following observations.
\begin{enumerate}
\item The bracket  $[\cdot ,...,\cdot ]  _{t}=f_{t}^{-1}\circ [\cdot ,...,\cdot ] _{1} \circ f_{t}\times f_{t}$
satisfies Nambu identity. Thus, when $t$
tends to $0$ the condition remains satisfied.

\item  The linear map $f_{t}$\ is invertible when $t\neq 0$ \ and may be
singular when $t=0.$ Then, we may obtain by degeneration a new $n$-Lie algebra.

\item The definition of formal degeneration may be extended naturally to infinite dimensional case.

\item  When $\corps$ is the complex field, the multiplication given by the limit, follows 
from a limit of the structure constants,
using the metric topology. In fact, $f_{t}\cdot [\cdot ,...,\cdot ]_1 $ 
corresponds to a change of basis when $t\neq 0$. When $t=0$, they give eventually a new
point in $\mathcal{L}ie^n_{m}$.

\item  If $f_t$ is defined by a power serie the images  
of $%
f_{t}\cdot \N$ are in general in the Laurent power series ring $\N\left[
\left[ t,t^{-1}\right] \right] $.  But when the
degeneration exists, it lies in the power series ring $\N\left[ \left[ t\right]
\right] $. 
\item Every formal degeneration is a global degeneration.
\end{enumerate}

\begin{remark}
Rigid $n$-Lie algebras will  have a special interest,
an open orbit of a given $n$-Lie algebra is dense in the irreducible component
in which it lies. Then, its Zariski closure determines an irreducible component of $\mathcal{L}ie^n_{m}$, i.e. 
all $n$-Lie algebras in this irreducible component are degenerations of the rigid $n$-Lie algebra and there is no $n$-Lie
algebra which degenerates to the rigid $n$-Lie algebra. Two non-isomorphic rigid $n$-Lie algebras correspond to different
irreducible components. So the number of rigid $n$-Lie algebra classes gives a lower bound of the number of
irreducible components of $\mathcal{L}ie^n_{m}$. Note that not all irreducible components 
are Zariski closure of open orbits.
\end{remark}

\section{$n$-Lie-Poisson algebras and Quantization}

\subsection{ $n$-Lie-Poisson algebras}We introduce the notion of  $n$-Lie-Poisson algebra.

\begin{definition}
An  \emph{ $n$-Lie-Poisson algebra} is a triple $(\N,\mu,\{.,.,.\})$ consisting of  a $\mathbb{K}$-vector space $N$, a bilinear map $\mu:\N\times \N\rightarrow \N$ and an $n$-ary bracket  $\{\cdot,...,\cdot\}$ such that
\begin{enumerate}
  \item $(\N,\mu)$ is a binary commutative associative algebra,
  \item $(\N,\{\cdot,...,\cdot \})$ is a $n$-Lie algebra,
 \item the following Leibniz rule
 \begin{align*}
\{x_{1},....,x_{n-1},\mu(x_{n},x_{n+1})\}=\mu(x_{n},\{x_{1},...,x_{n-1},x_{n+1}\})+\mu(\{x_{1},....,x_{n}\},x_{n+1})
\end{align*}
\end{enumerate}
holds for  all $x_{1},...,x_{n+1}\in \N$.
\end{definition}

A morphism of $n$-Lie-Poisson algebras is a linear map that is a morphism of the underlying $n$-Lie algebras and associative algebras.

\begin{example}
Let $C^{\infty}(\mathbb{R}^{3})$ be the algebra of $C^{\infty}$ functions  on $\mathbb{R}^{3}$ and $x_{1},x_{2},x_{3}$ the coordinates on $\mathbb{R}^{3}$. We define the ternary brackets as in \eqref{jacobian}, then $(C^{\infty}(\mathbb{R}^{3}),\{.,.,.\})$ is a ternary 3-Lie algebra. In addition the bracket satisfies the Leibniz rule: $\{f g,f_{2},f_{3}\}=f \{g,f_{2},f_{3}\}+\{f,f_{2},f_{3}\} g$ where $f,g,f_2,f_3\in C^{\infty}(\mathbb{R}^{3})$ and  the multiplication being the pointwise multiplication  that is  $fg(x)=f(x) g(x)$. Therefore,  the algebra is a 3-Lie-Poisson algebra.\\
This algebra was considered already in 1973 by Nambu \cite{Nambu} as a possibility of extending the Poisson bracket of standard hamiltonian mechanics to bracket of three functions defined by the Jacobian. Clearly, the Nambu bracket may be generalized further to an $n$-Lie-Poisson allowing for an arbitrary number of entries.\\
\end{example}

\subsection{Quantization of Nambu Mechanics}
The quantization
problem of Nambu Mechanics was  investigated by Dito, Flato, Sternheimer and Takhtajan  \cite{Flato,Flato2}, see also \cite{CZ1,CZ2,CZ3}. Let $M$ be an $m$-dimensional $\mathbb{C} ^\infty$-manifold and  $\A$ be the algebra of smooth real-valued functions on $M$. 

Assume that $\A$ carries a structure of $n$-Lie-Poisson structure, where the commutative associative multiplication is the pointwise multiplication.  The skew-symmetry of the Nambu bracket and the Leibniz identity imply that there exists an $n$-vector field $\eta$ on $M$ such that 
\begin{equation}\label{NambuTensor}
\{f_1,....,f_n\}=\eta (df_1,...,df_n), \quad \forall f_1,....,f_n\in\A.
\end{equation}
An $n$-vector field is called a Nambu tensor if its associated Nambu bracket defined by \eqref{NambuTensor} satisfies the Nambu identity \eqref{NambuIdentity}.

\begin{definition}
 A Nambu-Poisson manifold $(M, \eta)$ is a manifold $M$ on which is defined a Nambu tensor $\eta$. Then $M$ is said to be endowed with a Nambu-Poisson structure.
\end{definition}

The dynamics associated with a Nambu bracket on $M$ is specified by $n - 1$ Hamiltonians $H_1,..., H_{n-1}\in\A$ and the time evolution of $f \in \A$ is given by 
\begin{equation}\label{Hamiltonians}
\frac{df}{dt} =\{H_1,...,H_{n-1},f\}.
\end{equation}

Then $f\in  \A$  is called an integral of motion for the system defined by \eqref{Hamiltonians} if it satisfies $\{H_1,...,H_{n-1},f\} =0$.

It follows from the Nambu identity  that a Poisson-like theorem exists for Nambu-Poisson manifolds: 
\begin{theorem}
The Nambu bracket of $n$ integrals of motion is also an integral of motion.
\end{theorem}

It turns out that a direct application of deformation quantization to Nambu- Poisson structures is not possible, a solution to the quantization
problem was  presented in the approach of Zariski quantization of fields
(observables, functions, in this case polynomials). Instead of looking at the deformed Nambu bracket as some skew-symmetrized form of an $n$-linear product,  the Nambu bracket is deformed directly. 

In the case of previous example, the usual Jacobian bracket  is replaced  by any $n$-ary bracket  having the preceding properties, we get a "modified Jacobian" which is still a Nambu bracket. That is to say, the "modified Jacobian" is skew- symmetric, it satisfies the Leibniz rule with respect to the new bracket and the Nambu identity is verified. 

The deformed bracket is given by 
\begin{equation*}
\lbrack f_{1},f_{2},f_{3}]=\underset{\sigma \in \emph{S}_{3}}{\sum
\varepsilon (}\sigma )\frac{\partial f_{1}}{\partial x_{\sigma _{1}}}\times
\frac{\partial f_{2}}{\partial x_{\sigma _{2}}}\times \frac{\partial f_{3}}{%
\partial x_{\sigma _{3}}},
\end{equation*}
where $S_3$ is the permutation group of $\{1,2,3\}$ and  $\varepsilon (\sigma )$ is the signature. In this approach the whole problem of quantizing Nambu-Poisson structure reduces to the construction of the deformed product $\times$. A non-trivial abelian deformation of the algebra of polynomials on $\mathbb{R}^m$ doesn't exist because of the vanishing of the second Harrison cohomology group. Nevertheless,  it is possible to construct an abelian associative deformation of the usual pointwise product of the following form
\begin{equation}\label{DeformProduct}
f\times _{\beta }g=T(\beta (f)\otimes \beta (g)),
\end{equation}%
where $\beta$ maps a real polynomial on $\mathbb{R}^3$ to the symmetric algebra constructed over the
polynomials on $\mathbb{R}^3$ $(\beta :\mathcal{A}\rightarrow Symm(\mathcal{A}))$. $T$ is an "evaluation map" which allows to go back to (deformed)
polynomials $(T:Symm(\mathcal{A})\longrightarrow \mathcal{A})$.

It replaces the (symmetric) tensor product  by a symmetrized form of
a "partial" Moyal product on $\mathbb{R}^3$ (Moyal product on a hyperplane in $\mathbb{R}^3$ with deformation
parameter $t$). The extension of the map $\beta$ to deformed polynomials by requiring
that it annihilates (non-zero) powers of $t$, will give rise to an Abelian deformation of
the usual product ($T$ restores a $t$-dependence). In general \eqref{DeformProduct} does not define an associative
product and we look for a $\beta$ which makes the product $\times_\beta$ associative.

\subsection{Ternary Virasoro-Witt
algebras}
Curtright, Fairlie and Zachos provided the following 
ternary $q$-Virasoro-Witt
algebras constructed through the use of $su(1,1)$ enveloping algebra
techniques.

\begin{definition}
 The ternary algebras defined on the linear space $VW$ generated by
 $\{ Q_n,R_n\}_{n\in\mathbb{Z}}$ and  the
skewsymmetric ternary brackets:
\begin{eqnarray}\label{CFZbracket1}
[ Q_k,Q_m,Q_n]&=& (k-m)(m-n)(k-n) R_{k+m+n}\\
\label{CFZbracket2} [Q_k,Q_m,R_n]&=& (k-m)(Q_{k+m+n}+ z n R_{k+m+n})\\
\label{CFZbracket3} [Q_k,R_m,R_n]&=& (n-m) R_{k+m+n}\\
\label{CFZbracket4} [R_k,R_m,R_n]&=& 0
\end{eqnarray}
is called \emph{ternary Virasoro-Witt algebras}.
\end{definition}

Actually the previous ternary algebra is a ternary Nambu-Lie algebra
only in the cases $z=\pm 2 \imath$.

T. A. Larsson showed in \cite{LarssonTA} that the above ternary
Virasoro-Witt algebras can be constructed by applying,  to the
Virasoro representation acting scalar densities (i.e. primary
fields), the ternary commutator bracket
\begin{eqnarray}\label{NambuCommutator1}
[x,y,z]&=&x\cdot [y,z]+y\cdot [z,x]+z\cdot [x,y]\\
\ &=&x\cdot(y\cdot z)-x\cdot(z\cdot y)+y\cdot(z\cdot
x)-y\cdot(x\cdot z)+z\cdot(x\cdot y)-z\cdot(y\cdot x)\nonumber
\end{eqnarray}
where the dot denotes the associative multiplication and $[\cdot,\cdot]$ the binary
commutator bracket of its corresponding Lie algebra. He considered the operators
\begin{eqnarray*}
  E_m &=& e^{\imath m x} \\
  L_m &=& e^{\imath m x}(-\imath \frac{d}{dx}+\lambda m) \\
  S_m &=& e^{\imath m x}(-\imath \frac{d}{dx}+\lambda m)^2
\end{eqnarray*}
which lead to the binary commutators
$$[L_m,L_n]=(n-m)L_{m+n},\quad [E_m,E_n]=n E_{m+n},\quad
[L_m,E_n]=0.
$$
 Therefore, one obtains  the ternary brackets
\begin{eqnarray}
 [L_k,L_m,L_n]&=& (\lambda-\lambda ^2)(k-m)(m-n)(n-k) E_{k+m+n}\\
 \ [L_k,L_m,E_n]&=& (m-k)(L_{k+m+n}+(1- 2\lambda) n E_{k+m+n})\\
 \ [L_k,E_m,E_n]&=& (m-n) E_{k+m+n}\\
\ [ E_k,E_m,E_n]&=& 0
\end{eqnarray}
The brackets involving S's are not needed to recover the ternary
Virasoro-Witt algebras. The brackets
(\ref{CFZbracket1}-\ref{CFZbracket4}) are obtained by taking
$$L_m=-\sqrt[4]{\lambda(1-\lambda )} Q_k, \quad
E_m=(\sqrt[4]{\lambda(1-\lambda )})^{-1} R_k, \quad z=\frac{1-2
\lambda}{\sqrt{\lambda(1-\lambda )}}.
$$
Naturally, these ternary algebras are  3-Lie algebras
only for $\lambda=\pm 2 \imath$.
\begin{remark}
One may notice that the ternary commutator \eqref{NambuCommutator1}
does not lead automatically to ternary Nambu-Lie algebra when
starting from an associative algebra and the corresponding Lie
algebra given by the binary commutators. See \cite{Awata} and
\cite{ams:ternary} for triple commutator leading to
3-Lie algebras and ternary Hom-Nambu-Lie algebras \cite{AM2008}. More general construction of $(n+1)$-Lie algebras induced by $n$-Lie algebras was studied in
 \cite{AMS2}
\end{remark}


\begin{thebibliography}{99}

\bibitem{Kerner7}  Abramov V.,   Le Roy B.  and Kerner. R., \textit{Hypersymmetry : a Z3-graded generalization of supersymmetry},  Journ. in Math. Phys., \textbf{38} (3), (1997), 1650-1669.

\bibitem{ams:MabroukTernary} Ammar F.,  Makhlouf A. and  Silvestrov S.,\emph{ Ternary q-Virasoro-Witt Hom-Nambu-Lie algebras,} Journal of Physics A-Mathematical and Theoretical, \textbf{43},no. 26, 265204, (2010).



\bibitem{ams:MabroukRep}  Ammar F.,  Mabrouk S.,  Makhlouf A.,\emph{ Representations and cohomology of n-ary multiplicative Hom-Nambu-Lie algebras}, - J. Geom.  Physics , \textbf{61}, no. 10, 1898-1913, (2011).

\bibitem{ams:ternary} Arnlind J.,  Makhlouf A.,  Silvestrov S., \emph{Ternary Hom-Nambu-Lie algebras induced by Hom-Lie algebras}, J. Math. Phys. \textbf{51}, 043515, 11 pp. (2010).

\bibitem{AMS2}  Arnlind J.,  Makhlouf A.,  Silvestrov S., \emph{Construction of $n$-Lie algebras and $n$-ary Hom-Nambu-Lie algebras.} J. Math. Phys. \textbf{52} , no. 12, 123502, 13 pp, (2011).


\bibitem{AM2007} Ataguema H. and Makhlouf A., \textit{Deformations of
ternary algebras}, J. Gen. Lie Theory and App. \textbf{1},
41--55 (2007).
\bibitem{AM2008} Ataguema H. and Makhlouf A., \textit{Notes on cohomologies of ternary algebras
of associative type},   Journal of Generalized Lie Theory and Applications \textbf{ 3 }   no. 3,  157--174,  (2009).

\bibitem{AM2008} Ataguema H.,  Makhlouf A. and Silvestrov S. \textit{
Generalization of $n$-ary Nambu algebras and beyond}, Journal of  Mathematical  Physics  \textbf{50},  no. 8, 083501, (2009).

\bibitem{Awata} Awata H., Li M., Minic D. and Yoneya T
\textit{On the quantization of Nambu brackets}, arXiv (1999).


\bibitem{BSZ} Bai R., Song G., Zhang Y., \emph{On classification of $n$-Lie algebras}, Front. Math.
China 6, 581--606 (2011).

\bibitem{BL2007}  Bagger ~J. and Lambert ~N., \textit{Gauge Symmetry and
Supersymmetry of Multiple M2-Branes},      ArXiv:0711.0955, (2007).

\bibitem{Bazunova} Bazunova N., Borowiec A., and Kerner R., \textit{%
Universal differential calculus on ternary algebras}, Lett. Math.
Phys. \textbf{67} (2004), 195-206.


\bibitem{Bremmer} M.~R.~Bremmer and L.~A.~Peresi, \textit{Ternary analogues
of Lie and Malcev algebras}, Linear Algebra and its Applications 414
(2006) 1--18.


\bibitem{BMP}  Bordemann M., Makhlouf A. et Petit T. {\it D\'{e}formation par quantification et rigidit\'{e} des
alg\`{e}bres enveloppantes},  J. Algebra \textbf{285}, no. 2 (2005), 623--648.

\bibitem{Carles78}  Carles R. {\it Rigidit\'{e} dans la vari\'{e}t\'{e} des alg\`{e}bres}, CRASc Paris t
\textbf{286} (1978), pp 1123--1226.

\bibitem{CZ1} Curtright, T. L., Zachos, C. K. \textit{Branes, strings, and odd quantum Nambu brackets. }Quantum theory and symmetries, 206--217, World Sci. Publ., Hackensack, NJ, 2004. 

\bibitem{CZ2}  Zachos C.  Curtright, T.  \textit{Branes, quantum Nambu brackets and the hydrogen atom.} Czechoslovak J. Phys. \textbf{54 }  no. 11 (2004), 1393--1398.

\bibitem{CZ3} Curtright T. L., Zachos C. K.,  \textit{Nambu dynamics, deformation quantization, and superintegrability. } Superintegrability in classical and quantum systems, 29--46, CRM Proc. Lecture Notes, \textbf{37}, Amer. Math. Soc., Providence, RI, 2004.

\bibitem{Duplij} Borowiec. A, Dudek W. A and Duplij. S.
\textit{Basic concepts of ternary Hopf algebras,} Journal of Kharkov
National University, ser. Nuclei, Particles and Fields, V. 529  - N
3(15), (2001), 21--29.

\bibitem{Carlsson} Carlsson R., \textit{$N$-ary algebras} Nagoya Math. J.
\textbf{78} (1980), 45-56.

\bibitem{Carlsson1} Carlsson R., \textit{Cohomology of associative triple
systems}, Proc. Amer. Math. Soc. \textbf{60} (1976), 1-7.

\bibitem{CassasLodayPirashvili}  Cassas J.M., Loday J.-L. and Pirashvili  \textit{Leibniz
 $n$-algebras},
 Forum Math. \textbf{14} (2002), 189-207.




\bibitem{De Azcarraga2} De Azcarraga J.~A., Izquierdo J.~M., and
Perez J.~C. Bueno, \textit{On the higher-order generalizations of
Poisson structures}, J. Phys. A: Math. Gen. \textbf{30} (1997),
607-616.

\bibitem{De Azcarraga} De Azcarraga J.~A., Perelomov A.~M.., and
Perez  Bueno J.~C., \textit{The Schouten-Nijenhuis bracket,
cohomology and generalized Poisson structures}, arXiv:hep-th/9605067
(1996).

\bibitem{De Azcarraga1} De Azcarraga J.~A. and ~Perez Bueno J.~C., \textit{%
Multibracket simple Lie algebras}, In \textquotedblleft Physical
Applications and Mathematical Aspects of Geometry, Groups and
Algebra\textquotedblright , Vol I, pp. 103-107. World Sci., 1997.

\bibitem{De Azcarraga3} De Azcarraga J.~A., Izquierdo J.~M., \emph{$n$-ary algebras: a review with applications,} Journal of Physics A, Mathematical and Theoretical 07/2010; 43. DOI: 10.1088/1751-8113/43/29/293001.



\bibitem{Dito1}  Dito G., Flato M., Sternheimer D. and Takhtajan
L., \textit{Deformation Quantization and Nambu Mechanics}, Commun.
Math. Phys., \textbf{183}, 1-22. (1997)


\bibitem{Fialowski86}  Fialowski A. {\it Deformation of Lie algebras}, Math USSR sbornik,
vol. 55 n 2 (1986), pp 467-473.

\bibitem{Fialowski88}  Fialowski A. {\it An example of formal deformations
of Lie algebras, }In Deformation theory of algebras and structures and
applications, ed. Hazewinkel and Gerstenhaber, NATO {\it \ }Adv. Sci. Inst.
Serie C 297, Kluwer Acad. Publ. (1988).

\bibitem{Fialowski90}  Fialowski A. and O'Halloran J., {\it A comparison of
deformations and orbit closure}, Comm. in Algebra 18 (12) (1990) 4121-4140.

\bibitem{Fialowski99}  Fialowski A. and Post G. , {\it versal deformations
of Lie algebra }$L_{2}$, J. Algebra 236 n 1 (2001), 93-109.

\bibitem{Fialowski00}  Fialowski A. and Fuchs D. , {\it Construction of Miniversal deformations
of Lie algebras }, J. Funct. Anal 161 n 1(1999), 76-110.

\bibitem{Fialowski02}  Fialowski A. and Schlichenmaier M. {\it Global
deformation of the Witt algebra of Krichever-Novikov type}. Comm. Math. Phys. 260 (2005), no. 3, 579--612



\bibitem{Fialowski-Operad} Fialowski, A.; Mukherjee, G.; Naolekar, A.,  \emph{Versal deformation theory of algebras over a quadratic operad}. Homology Homotopy Appl. \textbf{16} (2014), no. 1, 179--198.

\bibitem{Flato} Dito G., Flato M., Sternheimer D., \emph{Nambu mechanics,$ n$-ary
operations and their quantization}, in: Deformation Theory and
Symplectic Geometry, Mathematical Physics Studies, vol. 20, Kluwer
Acad. Publ. 1997, 43-66.

\bibitem{Flato2} Dito G., Flato M., Sternheimer D.,  Takhtajan L.,\emph{Deformation Quantization and Nambu mechanics}, Commun. Math. Phys. 183, 1--22  (1997).

\bibitem{Filippov} Filippov V.~T.~, \textit{$n$-ary Lie algebras}, Sibirskii
Math. J. \textbf{24} (1985), 126-140 (Russian).

\bibitem{Gautheron1}  Gautheron P., \emph{Some Remarks Concerning Nambu Mechanics},
Letters in Mathematical Physics \textbf{37}: 103-116, 1996.

\bibitem{Gautheron}  Gautheron P., \emph{Simple facts concerning Nambu algebras}

\bibitem{Gerstenhaber64}  Gerstenhaber M. {\it On the deformations of rings
and algebras, }Ann. of Math 79, 84, 88 (1964, 66, 68).

\bibitem{Gerstenhaber-Schack} Gerstenhaber M. and  Schack S.D. {\it Relative Hochschild
cohomology, Rigid algebras and the Bockstein}, J. of pure and applied algebras 43
(1986), 53-74.

\bibitem{Gerstenhaber90}  Gerstenhaber M. and  Schack S.D.{\it Algebras,
bialgebras, Quantum groups and algebraic deformations, }Contemporary
mathematics Vol. 134, AMS (1992), 51-92.

\bibitem{Gerstenhaber98}  Gerstenhaber M. and Giaquinto A. {\it Compatible deformation }Contemporary
mathematics Vol. 229, AMS (1998), 159-168.

\bibitem{Gnedbaye1} Gnedbaye A.~V., \textit{Les alg\`{e}bres K-aires et
leurs op\'{e}rades}, C. R. Acad. Sci. Paris. s\'{e}rie I,
\textbf{321} (1995), 147-152.

\bibitem{Gnedbaye2} Gnedbaye A.~V., \textit{Op\'{e}rades, alg\`{e}bres de
Leibniz et triples de Jordan}, M\'{e}moire de synth\`{e}se pour
l'habilitation \`{a} diriger des recherches, Preprint IRMA
Strasbourg (2003).


\bibitem{GozeRemm}  Goze. N and Remm. R.,  \textit{ Dimension theorem for free ternary partially associative algebras and applications.} Journal of Algebra\textbf{ 348}, 14--36 (2011).


\bibitem{GozeRemm2} Goze, M.; Goze, N.; Remm, E., \emph{ $n$-Lie algebras}. Afr. J. Math. Phys. \textbf{8} (2010), no. 1, 17--28.

\bibitem{GozeRaush} Goze. N and Rausch de Traubenberg M.,  \textit{ Hopf algebras for ternary algebras}, J. Math. Phys. 50 (2009), no. 6, 063508.


\bibitem{Grabowski}  Grabowski J. and Marmo G., \emph{Remarks on Nambu-Poisson and
Nambu-Jacobi brackets}, J. Phys. A: Mah. Gen. \textbf{32}, (1999),
4239-4247.

\bibitem{Grabowski1}  Grabowski J., \emph{Abstract Jacobi and Poisson structures},
J. Geom. Phys. \textbf{9} (1992), 45-73.


\bibitem{Grun_OHallo} Grunewald F. and O'Halloran J., {\it A characterization of orbit closure and
applications,} Journal of Algebra, 116 (1988), 163-175.



\bibitem{Hanlon} Hanlon Ph. and  Wachs M., \textit{On Lie $k$-algebras}, Adv.
Math. \textbf{113} (1995), 206-236.

\bibitem{Harris}  Harris. B, \textit{Cohomology of Lie triple systems and
Lie algebras with involution}, Trans. Amer. Math. Soc
\textbf{98}(1961), 148-162.


\bibitem{Hestenes}  Hestenes, M. R. \textit{On ternary algebras}, Scripta
Math. \textbf{29}, no. 3-4, 253-272 (1973).

\bibitem{Hu}  Hu, N.,: \emph{$q$-Witt algebras,
$q$-Lie algebras, $q$-holomorph structure and
representations,}  Algebra Colloq. {\bf 6} ,
no. 1, 51--70 (1999).
\bibitem{Ibanez}  Ibanez R.,  de Leon M. J. C. Marrero and Martin de Diego D.,
\emph{Dynamics of generalized Poisson and Nambu-Poisson brackets},
J. Math. Phys. \textbf{38} (5) (1997), 2332-2344.
\bibitem{I_W} Inonu E. and Wigner E.P., {\it On the contraction of groups and their
representations,} Proc. Nat. Acad. Sci U.S., 39 (1953), 510-524.


\bibitem{Jacobson}  Jacobson. N. \textit{Lie and Jordan triple systems},
Amer. J. Math. \textbf{71}, (1949). 149-170.


\bibitem{KapranovGelfandZelinski} Kapranov M., Gelfand M. and Zelevinskii A.,
 \textit{Discriminants, Resultants and Multidimensional Determinants},
 Berlin Birkhauser (1994).


\bibitem{Kasymov1} Kasymov, Sh. M., \emph{On a theory of $n$-Lie algebras}. Algebra and Logic, 26, 155--166 (1987)


\bibitem{Katsylo} Katsylo P. and Mikhailov D., \textit{Ternary quartics and
3-dimensional commutative algebras}, arXiv:Alg-geom/9408900 (1994).

\bibitem{Kerner} Kerner R., \textit{Ternary algebraic structures and their
applications in physics},  in the ``Proc. BTLP 23rd International
Colloquium on Group Theoretical Methods in Physics'', ArXiv
math-ph/0011023, (2000).


\bibitem{Kerner2} Kerner R., \textit{Z3-graded algebras and non-commutative gauge theories}, dans le
livre "Spinors, Twistors, Clifford Algebras and Quantum
Deformations", Eds. Z. Oziewicz, B. Jancewicz, A. Borowiec, pp.
349-357, Kluwer Academic Publishers (1993).

\bibitem{Kerner3}  Kerner. R., \textit{Z3-grading and ternary algebraic structures}, dans les Proceedings
du Workshop "New Symmetries and Differential Geometry", Clausthal
1993, V. Dobrev, M.D. Doebner and S. Ushveridze eds., pp. 375-394,
World Scientific (1994).

\bibitem{Kerner4}  Kerner. R., \textit{The cubic chessboard : Geometry and physics, Classical Quantum Gravity} 14 (1997), pp. A203-A225.

\bibitem{Kerner5}  Kerner. R., \textit{Z3-graded ternary algebras, new gauge theories and quarks}, dans
Proceedings du Workshop "Topics in Quantum Field Theory", Maynooth
1995, T. Tchrakian, ed. World Scientific, (1995), 113-126.

\bibitem{Kerner6}  Kerner. R. and  Vainerman L., \textit{ On special classes of $n$-algebras},
Journ. in Math. Phys., 37 (5), 2553-2565, (1996).


\bibitem{Kerner9}  Kerner. R., \textit{Generalized Cohomologies and Differentials of Higher Order}, Proceedings de la conference DGMTP de Tianjin (Chine),(2006). Eds. Ge, Wang, World Scientific.




\bibitem{Kurosh}  Kurosh. A. G, \textit{Multioperator rings and algebras},
Russ Math Surv, 1969, \textbf{24} (1), 1-13.



\bibitem{LarssonTA} Larsson T. A.,
\textit{Virasoro 3-algebra from scalar densities}, arXiv:0806.4039 (2008).



\bibitem{Lister} W.~G.~Lister, \textit{Ternary rings}, Trans. Amer. Math.
Soc. \textbf{154} (1971), 37-55.

\bibitem{LodayValete-Operad} Loday J-L, and Vallette B., \emph{Algebraic operads,} A Seies of the Comprehensive Studies in Mathematics, Springer, 2012.


\bibitem{Loos}  Loos O.,  \textit{Assoziative tripelsysteme},
 Manuscripta Math. \textbf{7} (1972), 103-112.

\bibitem{loosSym}  Loos O.,  \textit{Symmetric spaces},
 vol. \textbf{1}, W. A. Benjamin, New York (1969), 103-112.


\bibitem{Makhlouf93}  Makhlouf A. {\it The irreducible components of the
nilpotent associative algebras }Revista Mathematica de la universidad
Complutence de Madrid Vol 6 n$.$1, (1993).

\bibitem{Makhlouf-Goze96}  Makhlouf A. and Goze M. {\it Classification of
rigid algebras in low dimensions}. Coll. Travaux en cours (M. Goze ed.) edition
Hermann (1996).

\bibitem{Makhlouf97}  Makhlouf A. {\it Alg\`{e}bre associative et calcul
formel }, Theoretical Computer Science 187  (1997), 123-145.



\bibitem{Michor1} Michor P.~W.~ and Vaisman I.~, \textit{A note on n-ary
Poisson brackets}, Proceedings of the 19th Winter School "Geometry
and Physics" (Srn\'{\i}, 1999). Rend. Circ. Mat. Palermo (2) Suppl.
No. 63 (2000), 165-172

\bibitem{Michor} Michor P.~W. and Vinogradov A.~ M., \textit{n-ary Lie and
assciative algebras}, , Rend. Sem. Mat. Univ. Pol. Torino.
\textbf{54} (1996), 373-392.


\bibitem{Nadaud1}  Nadaud F.{\it Generalized deformations, Koszul resolutions, Moyal products}, Rev. Math. Phys.
10(5) (1998), 685-704.

\bibitem{Nambu} Nambu Y., \textit{Generalized Hamiltonian mechanics}, Phys.
Rev. \textbf{D7} (1973), 2405-2412.

\bibitem{Neretin} Neretin Yu.A. {\it An estimate for the number of parameters defining an
$n$-dimensional algebra } Math USSR-Izv. \textbf{30} (2) (1988), 283--294.


\bibitem{NR} Nijenhuis, A. and Richardson J.R. {\it Cohomology and deformations in
graded Lie algebras} Bull. Amer. Math. Soc 72 (1966), 1-29 



\bibitem{Okubo}  Okubo S. \textit{Triple products and Yang-Baxter equation
(I): Octonionic and quaternionic triple systems}, J. Math.Phys. 34
(1993) 3273-3291.


\bibitem{pincson}  Pincson G.,{\it Noncommutative deformation theory}, Lett. Math. Phys. 41 (1997), 101-117.

\bibitem{Rausch1}  Rausch de Traubenberg. M. and Slupinski M. J. Slupinski, \textit{Finite-dimensional Lie algebras of order F}, J. Math. Phys. 43, (2002), 5145-5160.


\bibitem{Rausch3}  Rausch de Traubenberg., \textit{Ternary algebras and groups}. J. Phys.: Conf. Ser. 128,  (2008).

\bibitem{Rausch3}  Rausch de Traubenberg., \textit{Some Results on Cubic
and Higher Order Extensions of the Poincar\'e Algebra}, ArXiv:0811.1465, (2008).

\bibitem{Santana} Santana A. E. and Muradian R.,
\emph{Hopf Structure in Nambu-Lie n-Algebras}, Theoretical and
Mathematical Physics, \textbf{114} (1), (1998).

\bibitem{schliss1} Schlessinger M. {\it Functors of Artin rings, } Trans. Amer. Math. Soc.
130 (1968),  208-222.

\bibitem{Sokolov} Sokolov N. P., \emph{Introduction to the theory of
Multidimensional Matrices}, Kiev Naukova Dumaka(1972).

\bibitem{Takhtajan} Takhtajan L., \emph{On foundation of the generalized Nambu
mechanics}, Comm. Math. Phys. \textbf{160 }(1994), 295-315.

\bibitem{Takhtajan1} Takhtajan L., \emph{A higher order analog of
Chevally-Eilenberg complex and deformation theory of n-algebras},
St. Petersburg Math. J. \textbf{6} (1995), 429-438.

\bibitem{Yamaguti} Yamaguti. K. \textit{On the cohomologie space of Lie
triple
systems}, Kumamoto J. Sci. Ser. A 5 (1960), 44-52.
\bibitem{weibel} Weibel Ch.A.  \textit{An introduction to homological algebra}. Cambridge studies in advanced mathematics \textbf{38}, Cambridge University Press 1994.


\end{thebibliography}
\end{document}